\documentclass[11pt]{article}
\usepackage{amsmath,amssymb,amsfonts,amsthm,graphicx,color}
\usepackage{mathrsfs}
\usepackage{hyperref}
\usepackage{mathtools}
\usepackage{graphicx,type1cm,eso-pic,color}
\allowdisplaybreaks
\title{Lower and upper bounds for the explosion times\\ of a system of semilinear SPDEs}
\author{S. Sankar${}^{1}$, Manil T. Mohan${}^{2}$,  and S. Karthikeyan${}^{1,}\thanks{Corresponding author email: karthi@periyaruniversity.ac.in}$\\
	\footnotesize{$^1$Department of Mathematics, Periyar University, Salem 636 011, India}\\
	\footnotesize{$^2$Department of Mathematics, Indian
		Institute of Technology Roorkee, Roorkee 247 667, India}}
\date{}
\allowdisplaybreaks

\let\originalleft\left
\let\originalright\right
\renewcommand{\left}{\mathopen{}\mathclose\bgroup\originalleft}
\renewcommand{\right}{\aftergroup\egroup\originalright}

\begin{document}
	\maketitle \setcounter{page}{1} \numberwithin{equation}{section}
	\newtheorem{theorem}{Theorem}[section]
	\newtheorem{lemma}{Lemma}[section]
	\newtheorem{Pro}{Proposition}[section]
	\newtheorem{Ass}{Assumption}[section]
	\newtheorem{Def}{Definition}[section]
	\newtheorem{Rem}{Remark}[section]
	\newtheorem{corollary}{Corollary}[section]
	\newtheorem{proposition}{Proposition}[section] 
	\newtheorem{ack}{Acknowledgement:}
	\begin{abstract}
		In this paper, we obtain lower and upper bounds for the blow-up times for a system of semilinear stochastic partial differential equations.  Under suitable assumptions, lower and upper  bounds of the explosive times are obtained by using explicit solutions of an associated system of random partial differential equations and a formula due to Yor. We provide lower and upper bounds for the probability of finite-time blow-up solution as well. The above obtained results are also extended for semilinear SPDEs forced by two dimensional Brownian motions.
		
	\end{abstract}

		\baselineskip 17pt
	\section{Introduction}
	Stochastic partial differential equations (SPDEs) play a prominent role in describing the evolution of many complex dynamical systems that are driven by random perturbations. An extensive range of physical phenomena, encountered in statistical mechanics, mathematical physics, mathematical biology,  fluid dynamics and mathematical finance can be effectively modelled by SPDEs. As there are only limited methods available to obtain standard analytical solutions to SPDEs, the problem of existence, uniqueness, stability and other properties of the solutions have been the focus of many researchers in recent years, cf. \cite{chen2017,chues2000, denis2005, doz2020, gess, Eug2017, Eug2019}, etc. and the references therein. Under suitable conditions, the problem of existence and uniqueness of global solutions to nonlinear SPDEs have been extensively investigated by many authors, for example, see \cite{chow2015,gyong2000}, etc.

	In the seminal paper \cite{fuji1966}, Fujita proved that  the following semilinear heat equation defined over a bounded smooth domain $D\subset\mathbb{R}^d$, $d\geq 1$
	\begin{equation*}
		\frac{\partial u(t,x)}{\partial t}=\Delta u(t,x)+u^{1+\alpha}(t,x), \ x\in D,
	\end{equation*}	
	with the Dirichlet boundary condition, where $\alpha>0$ is a constant,  explodes in finite time for all  non-negative initial data $u(0,x)\in L^2(D)$ satisfying $\int_Du(0,x)\psi(x)dx>\lambda^{\frac{1}{\alpha}}$, $\lambda$ being the first eigenvalue of the Laplacian on $D$ and $\psi$, the corresponding normalized eigenfunction with $\|\psi\|_{L^1}=1$.
	After the basic papers by Kaplan \cite{kaplan} and Fujita \cite{fuji1966, Fuji1968}, many researchers have developed new ideas for blow-up of solutions and investigated the blow-up phenomena to a more general class of nonlinear partial differential equations (PDEs), which develops singularities in finite time. In particular, by adopting different methods, the problem of existence of blow-up solutions in a finite time and estimating lower and upper bounds for blow-up times have been addressed by many researchers including but not limited to  \cite{liao,sha2019,sat2019,sat2021}, etc.
	
	In contrast to the deterministic case, very few results (blow-up phenomena) are known for SPDEs due to some complexity in stochastic analysis and so one can try to resolve such questions for at least a few special cases, see \cite{doz2014,li,wang}, etc. Using the well-known Yor formula (cf.\cite{Fenman1997,yor1992,yor2001}), Dozzi and L\'{o}pez-Mimbela \cite{doz2010}  derived lower and upper bounds for the finite-time blow-up of positive solutions of a semilinear SPDE. Chow and Liu \cite{chow2012} proved that the semi-linear stochastic functional parabolic equations of retarded type explodes in finite time in $L^{p}$-norm. Kolkovska and L\'{o}pez-Mimbela \cite{car2013} estimated the lower and upper bounds for the blow-up time of a semilinear heat equation on a bounded domain. Dozzi, Kolkovska and L\'{o}pez-Mimbela \cite{doz2013} obtained  the lower and upper bounds for the blow-up times to a system of semilinear SPDEs of the form:
	\begin{equation*}
		\left\{
		\begin{aligned}
			du_{1}(t,x)&=\left[ \left(\Delta+V_{1}\right)u_{1}(t,x)+u^{p}_{2}(t,x) \right]dt+k_{1}u_{1}(t,x)dW_{t} ,\nonumber\\
			du_{2}(t,x)&=\left[ \left(\Delta+V_{2}\right)u_{2}(t,x)+u^{q}_{1}(t,x) \right]dt+k_{2}u_{2}(t,x)dW_{t}, 
		\end{aligned}
		\right.
	\end{equation*}
	for $x\in D,\ t>0$, with the Dirichlet boundary conditions. Here $p\geq q>1$ are constants, $D \subset \mathbb{R}^{d}$ is a bounded smooth domain, $V_{i}>0$ and $k_{i}\neq 0$ are constants, $i=1,2,$ and $ W_{t} $ is a standard one-dimensional Brownian motion defined on some probability space $\left( \Omega, \mathscr{F}, \mathbb{P} \right)$.

	Motivated by this paper, we consider the problem to obtain lower and upper bounds for the blow-up times of the following  system of semilinear SPDEs: 
	\begin{equation}\label{b1} 
	\left\{
	\begin{aligned}
	du_{1}(t,x)&=\left[ \left(\Delta+V_{1}\right)u_{1}(t,x)+u^{1+\beta_{1}}_{1}(t,x)+u^{1+\gamma_{2}}_{2}(t,x) \right]dt+k_{1}u_{1}(t,x)dW_{t}, \\
	du_{2}(t,x)&=\left[ \left(\Delta+V_{2}\right)u_{2}(t,x)+u^{1+\gamma_{1}}_{1}(t,x)+u^{1+\beta_{2}}_{2}(t,x) \right]dt+k_{2}u_{2}(t,x)dW_{t}, 
	\end{aligned}
	\right.
	\end{equation}
	for $x \in D,\ t>0$, along with the Dirichlet conditions
	\begin{equation}\label{b2}
		\left\{
		\begin{array}{ll}
			u_{i}(0,x)=f_{i}(x)\geq 0,  &x \in D, \\
			u_{i}(t,x)=0, \ \ & x \in \partial D,\ t\geq 0,\  i=1,2.
		\end{array}
		\right.
	\end{equation}
	Here $\beta_{1}\geq \gamma_{2}\geq \gamma_{1}\geq \beta_{2}>0$ are constants, $D \subset \mathbb{R}^{d}$ is a bounded smooth domain, $V_{i}>0$ and $k_{i}\neq 0$ are constants, $f_{i}$ are of class $C^{2}$ and not identically zero for $i=1,2$ and $W_{t} $ is a standard one-dimensional Brownian motion defined on some probability space $\left( \Omega, \mathscr{F}, \mathbb{P} \right)$. One can refer to Proposition 3.7 in \cite{gyong2000} and Theorem 9.15 in \cite{peszat} for existence of weak and mild solutions of  \eqref{b1}-\eqref{b2} and their equivalence. 
	
	The aim of this article is to derive estimates for upper and lower bounds to the blow-up times of the solution $u=(u_1,u_2)^{\top}$ of the system \eqref{b1}-\eqref{b2}. Using the explicit solutions of an associated system of random PDEs, the distribution  functions of several blow-up times are obtained. Also we obtain probabilistic bounds for the blow-up solution of the system \eqref{b1}-\eqref{b2} for $V_1=V_2=0$. Recall that $\lambda>0$ is the first eigenvalue of the Laplacian operator $-\Delta$ on $D$ and $\psi$ is the corresponding normalized  eigenfunction so that $\|\psi\|_{L^1}=1$. 
	
	The rest of the article is structured as follows: The next section is devoted to obtain an associated system of random PDEs, by using  a random transformation of the system \eqref{b1}-\eqref{b2}, which is useful to establish  lower and upper bounds for the blow-up time $\tau$. In sections   \ \ref{sec3} and \ref{sec4}, we obtain the lower and upper bounds for the blow-up times for the case $(1+\beta_{1})k_{1}-k_{1}=(1+\gamma_{2})k_{2}-k_{1} =:\rho_{1},\	(1+\gamma_{1})k_{1}-k_{2}=(1+\beta_{2})k_{2}-k_{2}=:\rho_{2}$. Under the above setting, the lower and upper bounds of the finite-time blow-up  time $\tau$ ($\tau_*$ and $\tau^*$) are obtained explicitly  (Theorems \ref{t2} and \ref{thm4.1}) in terms of the relevant exponential functionals of the form $\int_{0}^{t}e^{\rho_{i}W_{r}}dr$, $i=1,2$. The above condition has been relaxed in section \ref{sec5}, and the lower and upper bounds ($\tau_{**}$ and $\tau^{**}$) for the blow-up time $\tau$ of the solution to the system \eqref{b1}-\eqref{b2} is established (Theorems \ref{thm5.1} and \ref{thm5.2}). In section \ref{sec6}, by applying the Yor formula (cf. \cite{yor2005},\cite{revuz1999}\cite{yor2001}, etc.) and the method adopted in \cite{doz2010, Eug2017} and \cite{niu2012} with suitable parameters, we provide lower and upper bounds for the probability of blow-up solution of the system \eqref{b1}-\eqref{b2} with $V_1 = V_2 = 0$. An extension of the above results to  semilinear SPDEs driven by two-dimensional Brownian motions is given in section \ref{sec7}. 
	
	\section{A System of Random PDEs}\label{sec2}
	In this section, we obtain a system of random PDEs from the  system \eqref{b1}-\eqref{b2}  by using random transformations,
	\begin{eqnarray}\label{1}
		v_{i}(t,x)=\exp\left\{-k_{i}W_{t}\right\}u_{i}(t,x), \ i=1,2,
	\end{eqnarray}
	for $t\geq 0, \ x\in D$. 
	
Let  $\left\{ S_{t} \right\}$ be the semigroup of bounded linear operators defined by
	\begin{eqnarray}
		S_{t}f(x)=\mathbb{E}\left[ f(X_{t}), \ t<\tau_{D}|X_{0}=x \right], \quad  x\in D,\nonumber
	\end{eqnarray}
	for all bounded and measurable $f : D\rightarrow \mathbb{R},$ where $\left\{ X_{t} \right\}_{t\geq 0}$ is the $d$-dimensional Brownian motion with variance parameter $2$, killed at the time $\tau_{D}$ at which it strikes the boundary $\partial D.$ As discussed above,  $\lambda>0$ is the first eigenvalue of the Laplacian on $D$, which satisfies 
	\begin{eqnarray} \label{a3}
		-\Delta \psi(x)=\lambda \psi(x), \ \ x\in D,
	\end{eqnarray}
	$\psi$ being the corresponding eigenfunction, which is strictly positive on $D$ (see Corollary 3.3.7 in \cite{davies}) and $\psi|_{\partial D} =0.$ Remember that $$S_{t}\psi=\exp\{{-\lambda t}\}\psi, \ \ t\geq 0.$$ We assume that $\psi$ is normalized so that $\int_{D} \psi(x)dx=1.$
	
	By Ito's formula for one-dimensional Brownian motion (see \cite{evans,klebaner} and Theorem 2.10 in \cite{Gawarecki}), we get 
	\begin{eqnarray}
		\exp\{{-k_{i}W_{t}}\}=1-k_{i} \int_{0}^{t} \exp\{{-k_{i}W_{s}}\}dW_{s}+\frac{k_{i}^{2}}{2} \int_{0}^{t} \exp\{{-k_{i}W_{s}}\}ds. \nonumber 
	\end{eqnarray}
    For test functions $\varphi_i\in C^\infty_0(D),\ i=1,2, $ let us take $$u_{i}(t,\varphi_{i}):=\int_{D}u_{i}(t,x)\varphi_{i}(x)dx.$$ 
	
	Let us recall the notion of weak solution of $\eqref{b1}$. Let $\tau \leq  +\infty$ be a stopping time. A continuous random vector field $u=(u_1,u_2)^{\top}=\left\lbrace (u_{1}(t,x),u_2(t,x))^{\top}: \ t\geq 0,\ x \in D \right\rbrace $ is a weak solution of $\eqref{b1}$ on the interval $(0,\tau)$ provided that for every $\varphi_i\in C^\infty_0(D),$ there holds	
		\begin{align}\label{c1} 
	u_{i}\left(t,\varphi_{i}\right)&=u_{i}\left(0,\varphi_{i}\right)+\int_{0}^{t}u_{i}\left(s,(\Delta+V_{i})\varphi_{i}\right)ds+\int_{0}^{t}u^{1+\beta_{i}}_{i}\left(s,\varphi_{i}\right)ds\nonumber\\
	&\quad+\int_{0}^{t} u^{1+\gamma_{j}}_{j}\left(s,\varphi_{i}\right)ds+k_{i}\int_{0}^{t}u_{i}\left(s,\varphi_{i}\right)dW_{s},
	\end{align}
where $i=1,2,\ \left\lbrace j \right\rbrace=\left\lbrace1,2\right\rbrace/\{i\}$. By applying the integration by parts formula Chapter 8 in \cite{klebaner}, we obtain 
	\begin{align*}
	v_{i}(t,\varphi_{i}):&=\int_{D}v_{i}(t,x)\varphi_{i}(x)dx \nonumber\\
	&=v_{i}(0,\varphi_{i})+\int_{0}^{t} \exp\{{-k_{i}W_{s}}\}du_{i}(s,\varphi_{i})\nonumber\\
	&\quad+\int_{0}^{t} u_{i}(s,\varphi_{i})\left(-k_{i}  \exp\{{-k_{i}W_{s}}\}dW_{s}+\frac{k_{i}^{2}}{2}  \exp\{{-k_{i}W_{s}}\}ds\right)\nonumber\\
	&\quad+\left[\exp\{{-k_{i}W_{\cdot}}\},u_{i}(\cdot,\varphi_{i}) \right](t),\nonumber   
	\end{align*}
	where the quadratic variation (Chapter 8 in \cite{klebaner}) is given by
	\begin{eqnarray}
		\left[\exp\{{-k_{i}W_{\cdot}}\},u_{i}(\cdot,\varphi_{i}) \right](t)=-k^{2}_{i} \int_{0}^{t} \exp\{{-k_{i}W_{s}}\}u_{i}(s,\varphi_{i})ds, \ \ t\geq 0.\nonumber 
	\end{eqnarray}
	Therefore, it can be easily seen that 
	\begin{align} \label{c3}
		v_{i}(t,\varphi_{i})&=v_{i}(0,\varphi_{i})+\int_{0}^{t}\left[ v_{i}(s,(\Delta+V_{i})\varphi_{i})-\frac{k^{2}_{i}}{2}v_{i}(s,\varphi_{i}) \right]ds\nonumber\\
		&\quad+\int_{0}^{t} \exp\{{-k_{i}W_{s}}\}( \exp(k_{i}W_{s}) v_{i})^{1+\beta_{i}}(s,\varphi_{i})ds\nonumber\\
		&\quad+\int_{0}^{t} \exp\{{-k_{i}W_{s}}\}\left( \exp\{{k_{j}W_{s}}\} v_{j}\right)^{1+\gamma_{j}}(s,\varphi_{i})ds.  
	\end{align} 
	Hence, the vector $\left( v_{1}(t,x),v_{2}(t,x) \right)^{\top}$ is a weak solution of the system
	\begin{equation}\label{pde1}
		\left\{
		\begin{aligned}
			\frac{\partial v_{i}(t,x)}{\partial t}&=\left( \Delta+V_{i}-\frac{k^{2}_{i}}{2} \right)v_{i}(t,x)+\exp\{{-k_{i}W_{t}}\}\left( \exp\{{k_{i}W_{t}}\}v_{i}(t,x) \right)^{1+\beta_{i}}\\
			&\quad+\exp\{{-k_{i}W_{t}}\}\left( \exp\{{k_{j}W_{t}}\}v_{j}(t,x) \right)^{1+\gamma_{j}}, \\  v_{i}(0,x)&=f_{i}(x), \ \ i=1,2.
		\end{aligned}
		\right.
	\end{equation}
	The above system is understood in the pathwise sense and thus classical results for parabolic PDEs can be applied to show existence, uniqueness and positivity of solution $v=(v_1,v_2)^{\top}$ up to an eventual blow-up (see \cite{doz2013} and Friedman Chapter 7 in \cite{friedman1964}). 
	
	The integral form of \eqref{pde1} is given by (see Chapter 4 in \cite{pazy})
	\begin{align}\label{a2}
	v_{1}(t,x)&=\exp \left\{ t \left( V_{1}-\frac{k^{2}_{1}}{2}\right) \right\}S_{t}f_{1}(x)\nonumber\\
	&\quad+\int_{0}^{t}\exp\left\{{(t-r) \left(V_{1}-\frac{k^{2}_{1}}{2} \right)}\right\}S_{t-r}\left[ \exp\{{-k_{1}W_{r}}\}\left(\exp\{{k_{1}W_{r}}\}v_{1}(r,\cdot)\right)^{1+\beta_{1}} \right](x)dr\nonumber\\
	&\quad +\int_{0}^{t}\exp \left\{{(t-r)\left(V_{1}-\frac{k^{2}_{1}}{2}\right)}\right\}S_{t-r}\left[ \exp\{{-k_{1}W_{r}}\}\left( \exp\{{k_{2}W_{r}}\}v_{2}(r,\cdot)\right)^{1+\gamma_{2}} \right](x)dr,\nonumber\\
	v_{2}(t,x)&=\exp \left\{ t \left( V_{2}-\frac{k^{2}_{2}}{2}\right) \right\}S_{t}f_{2}(x)\nonumber\\
	&\quad+\int_{0}^{t}\exp\left\{{(t-r) \left(V_{2}-\frac{k^{2}_{2}}{2} \right)}\right\}S_{t-r}\left[ \exp\{{-k_{2}W_{r}}\}\left(\exp\{{k_{2}W_{r}}\}v_{2}(r,\cdot)\right)^{1+\beta_{2}} \right](x)dr\nonumber\\
	&\quad +\int_{0}^{t}\exp \left\{{(t-r)\left(V_{2}-\frac{k^{2}_{2}}{2}\right)}\right\}S_{t-r}\left[ \exp\{{-k_{2}W_{r}}\}\left( \exp\{{k_{1}W_{r}}\}v_{1}(r,\cdot)\right)^{1+\gamma_{1}} \right](x)dr,
	\end{align}
	for $t\geq 0,$ which is also a mild form of \eqref{pde1}.
	
	Due to \eqref{1} and the a.s. continuity of Brownian paths, it follows that if $\tau$  is the blow-up time of the system \eqref{pde1}  with the initial values of the above form, then $\tau$ is also the blow-up time for the system \eqref{b1}-\eqref{b2} (see Corollary 1 in \cite{doz2014}). Our aim is to find random times $\tau_{*}$ and $\tau^{*}$ such that $0\leq \tau_*\leq \tau\leq \tau^*$.

	\section{A Lower Bound for $\tau$}\label{sec3}
	In this section, we find a lower bound $\tau_*$ to the blow-up times such that $ \tau_*\leq \tau$.	First, we consider the equation \eqref{b1}-\eqref{b2} with parameters
	\begin{equation}\label{a4}
		\begin{aligned}
			V_{i}=\lambda+\frac{k^{2}_{i}}{2}, \ \ i=1,2, \ \   (1+\beta_{1})k_{1}-k_{1}&=(1+\gamma_{2})k_{2}-k_{1} =:\rho_{1},\\
			(1+\gamma_{1})k_{1}-k_{2}&=(1+\beta_{2})k_{2}-k_{2}=:\rho_{2}. 
		\end{aligned}
			\end{equation}
	 
	\begin{theorem}\label{t2} Assume that the conditions (\ref{a4}) hold, and let the initial values be of the form
		\begin{eqnarray} \label{q1}
			f_{1}=L_{1} \psi \ \ \mbox{and} \ \ f_{2}=L_{2} \psi ,
		\end{eqnarray} 
		for some positive constants $L_{1}$ and $L_{2}$ with $L_{1}\leq L_{2}$. Let $\tau_{\ast}$ be given by

		\begin{align}\label{e1}
		\tau_{\ast} = \inf \Bigg\{ t\geq 0 : &\int_{0}^{t} \exp\{{\rho_{1}W_{r}}\}dr
		\geq  \frac{1}{(\beta_{1}+\gamma_{2}+1)(L^{\beta_{1}}_{1}\left\|\psi \right\|_{\infty}^{\beta_{1}}+L^{\gamma_{2}}_{1}\left\|\psi \right\|_{\infty}^{\gamma_{2}})}, \nonumber\\
		\mbox{ or }&\int_{0}^{t} \exp\{{\rho_{2}W_{r}}\}dr
		\geq  \frac{1}{(\gamma_{1}+\beta_{2}+1)(L^{\gamma_{1}}_{2}\left\|\psi \right\|_{\infty}^{\gamma_{1}}+L^{\beta_{2}}_{2}\left\|\psi \right\|_{\infty}^{\beta_{2}})} \Bigg\},
		\end{align}
	where $\|\psi\|_{\infty}:=\sup\limits_{x\in D}\psi(x)$. 	Then $\tau_{\ast}\leq\tau$.
		\begin{proof}
			Let $v_{1}$ and $v_{2}$ solve (\ref{a2}). Then, we have
			\begin{align}
				v_{1}(t,x)&=\exp\{{\lambda t}\}S_{t}f_{1}(x)+\int_{0}^{t}\exp\{{\lambda (t-r)}\}S_{t-r} \left( \exp\{{\rho_{1} W_{r}}\}v^{1+\beta_{1}}_{1}(r,x) \right)dr\nonumber\\
				&\quad+\int_{0}^{t} \exp\{{ \lambda (t-r)}\}S_{t-r}\left( \exp\{{ \rho_{1} W_{r}}\}v^{1+\gamma_{2}}_{2}(r,x)\right)dr,\nonumber\\
				v_{2}(t,x)&=\exp\{{\lambda t}\}S_{t}f_{2}(x)+\int_{0}^{t}\exp\{{\lambda (t-r)}\}S_{t-r} \left( \exp\{{\rho_{2} W_{r}}\}v^{1+\gamma_{1}}_{1}(r,x) \right)dr\nonumber\\
				&\quad+\int_{0}^{t} \exp\{{ \lambda (t-r)}\}S_{t-r}\left( \exp\{{ \rho_{2} W_{r}}\}v^{1+\beta_{2}}_{2}(r,x)\right)dr, \ \ x\in D, \ \ t\geq 0. \nonumber
			\end{align}
			Let us define the operators $\mathcal{T}_{1}, \ \mathcal{T}_{2}$ as
			\begin{align}
			\mathcal{T}_{1}[v,w](t,x)&=\exp\{{\lambda t}\}S_{t}f_{1}(x)+\int_{0}^{t}\exp\{{\rho_{1} W_{r}+\lambda (t-r)}\}\left( S_{t-r}v \right)^{1+\beta_{1}}dr\nonumber\\
			&\quad+\int_{0}^{t}\exp\{{\rho_{1} W_{r}+\lambda (t-r)}\}\left( S_{t-r}w \right)^{1+\gamma_{2}}dr, \nonumber\\
			\mathcal{T}_{2}[v,w](t,x)&=\exp\{{\lambda t}\}S_{t}f_{2}(x)+\int_{0}^{t}\exp\{{\rho_{2} W_{r}+\lambda (t-r)}\}\left( S_{t-r}v \right)^{1+\gamma_{1}}dr\nonumber\\
			&\quad+\int_{0}^{t}\exp\{{\rho_{2} W_{r}+\lambda (t-r)}\}\left( S_{t-r}w \right)^{1+\beta_{2}}dr, \nonumber
			\end{align}
			where $v$ and $w$ are non-negative, bounded and measurable functions. 
			
			First we shall prove that \begin{eqnarray}
			v_{1}(t,x)=\mathcal{T}_{1}\left[ v_{2},v_{1} \right](t,x), \  v_{2}(t,x)=\mathcal{T}_{2}\left[ v_{1},v_{2} \right](t,x), \ \ x\in D, \ \ 0 \leq t<\tau_{*}.\nonumber
			\end{eqnarray}
			for some non-negative, bounded and measurable functions $v_1$ and $v_2$.
			
            We again recall that the references Proposition 3.7 in \cite{gyong2000} and Theorem 9.15 in \cite{peszat} for the existence of weak  and mild solutions and their equivalence. Moreover, on the set $t< \tau_{*},$ we fix
			\begin{align}
				\mathscr{G}_{1}(t)&=\left[ 1-\left(\beta_{1}+\gamma_{2}+1\right) \int_{0}^{t}  e^{\rho_{1} W_{r}} \left( \left\| \exp\{{\lambda r}\}S_{r}f_{1} \right\|_{\infty}^{\beta_{1}}+\left\| \exp\{{\lambda r}\}S_{r}f_{1} \right\|_{\infty}^{\gamma_{2}}\right)dr\right]^{\frac{-1}{\beta_{1}+\gamma_{2}+1}},\nonumber\\
				\mathscr{G}_{2}(t)&=\left[ 1-\left(\gamma_{1}+\beta_{2}+1\right) \int_{0}^{t} e^{\rho_{2} W_{r}} \left( \left\| \exp\{{\lambda r}\}S_{r}f_{2} \right\|_{\infty}^{\gamma_{1}}+\left\| \exp\{{\lambda r}\}S_{r}f_{2} \right\|_{\infty}^{\beta_{2}}\right)dr\right]^{\frac{-1}{\gamma_{1}+\beta_{2}+1}}.\nonumber
			\end{align}
         By \eqref{e1}, it is immediate that $\mathscr{G}_{1}(t)$ and $\mathscr{G}_{2}(t)$ are well defined. Then, it can be easily seen that 
			\begin{eqnarray}
				\frac{d \mathscr{G}_{1}(t)}{dt}=\exp\{{\rho_{1} W_{t}}\} \left( \left\| \exp\{{\lambda t}\}S_{t}f_{1} \right\|_{\infty}^{\beta_{1}}+\left\| \exp\{{\lambda t}\}S_{t}f_{1} \right\|_{\infty}^{\gamma_{2}}\right) \mathscr{G}^{\beta_{1}+\gamma_{2}+2}_{1}(t), \ \ \mathscr{G}_{1}(0)=1, \nonumber
			\end{eqnarray}
			so that 
			\begin{eqnarray}
			\mathscr{G}_{1}(t)=1+\int_{0}^{t} \exp\{{\rho_{1} W_{r}}\} \left( \left\| \exp\{{\lambda r}\}S_{r}f_{1} \right\|_{\infty}^{\beta_{1}}+\left\| \exp\{{\lambda r}\}S_{r}f_{1} \right\|_{\infty}^{\gamma_{2}}\right) \mathscr{G}^{\beta_{1}+\gamma_{2}+2}_{1}(r) dr. \nonumber
			\end{eqnarray}
			Similarly, we have 
			\begin{eqnarray}
			\mathscr{G}_{2}(t)=1+\int_{0}^{t} \exp\{{\rho_{2} W_{r}}\} \left( \left\| \exp\{{\lambda r}\}S_{r}f_{2} \right\|_{\infty}^{\gamma_{1}}+\left\| \exp\{{\lambda r}\}S_{r}f_{2} \right\|_{\infty}^{\beta_{2}}\right) \mathscr{G}^{\gamma_{1}+\beta_{2}+2}_{2}(r) dr. \nonumber
			\end{eqnarray}
			Let us choose $v,w_{1}\geq 0$ such that $$v(t,x)\leq \exp\{{\lambda t}\}S_{t}f_{1}(x)\mathscr{G}_{1}(t), \ \ w_{1}(t,x)\leq \exp\{{\lambda t}\}S_{t}f_{1}(x)\mathscr{G}_{1}(t),$$ for $x\in D$ and $t<\tau_{*}.$ Then $\exp\{{\lambda t}\}S_{t}f_{1}(x)\leq \mathcal{T}_{1}\left[ v,w_{1} \right](t,x)$ and
			\begin{align}
				\mathcal{T}_{1}[v,w_{1}](t,x)&=\exp\{{\lambda t}\}S_{t}f_{1}(x)+\int_{0}^{t}\exp\{{\rho_{1} W_{r}+\lambda (t-r)}\} \left( S_{t-r}v(r,x) \right)^{1+\beta_{1}}dr\nonumber \\
				& \quad +\int_{0}^{t}\exp\{{\rho_{1} W_{r}+\lambda (t-r)}\} \left( S_{t-r}w_{1}(r,x) \right)^{1+\gamma_{2}}dr \nonumber\\
				&\leq \exp\{{\lambda t}\}S_{t}f_{1}(x)+\int_{0}^{t}\exp\{{\rho_{1} W_{r}+\lambda (t-r)}\}\nonumber\\ 
				& \qquad\times\left[ \exp\{{\lambda r}\} \mathscr{G}_{1}(r)\exp\{{\beta_{1}\lambda r}\} \mathscr{G}^{\beta_{1}}_{1}(r) \left\|S_{r}f_{1}\right\|_{\infty}^{\beta_{1}}S_{t-r}(S_{r}f_{1}(x)) \right]dr \nonumber\\
				& \quad +\int_{0}^{t}\exp\{{\rho_{1} W_{r}+\lambda (t-r)}\}\nonumber\\&\qquad\times\left[ \exp\{{\lambda r}\}\mathscr{G}_{1}(r)\exp\{{\gamma_{2}\lambda r}\} \mathscr{G}^{\gamma_{2}}_{1}(r)\left\|S_{r}f_{1}\right\|_{\infty}^{\gamma_{2}}S_{t-r}(S_{r}f_{1}(x)) \right]dr \nonumber\\
				&= \exp\{{\lambda t}\}S_{t}f_{1}(x)\Bigg\{ 1+ \int_{0}^{t} \exp\{{\rho_{1} W_{r}}\}\left\| \exp\{{\lambda r}\}S_{r}f_{1} \right\|_{\infty}^{\beta_{1}} \mathscr{G}^{1+\beta_{1}}_{1}(r)dr\nonumber\\
				& \qquad+\int_{0}^{t} \exp\{{\rho_{1} W_{r}}\} \left\| \exp\{{\lambda r}\} S_{r}f_{1} \right\|_{\infty}^{\gamma_{2}} \mathscr{G}^{1+\gamma_{2}}_{1}(r)dr\Bigg\} \nonumber\\
				&\leq \exp\{{\lambda t}\} S_{t}f_{1}(x) \bigg[ 1+\int_{0}^{t}\exp\{{\rho_{1} W_{r}}\} \nonumber\\&\qquad\times\left( \left\| \exp\{{\lambda r}\} S_{r}f_{1} \right\|_{\infty}^{\beta_{1}}+\left\| \exp\{{\lambda r}\} S_{r}f_{1} \right\|_{\infty}^{\gamma_{2}}\right) \left( \mathscr{G}^{1+\beta_{1}}_{1}(r)+\mathscr{G}^{1+\gamma_{2}}_{1}(r)\right)dr \bigg] \nonumber\\
				&\leq \exp\{{\lambda t}\}S_{t}f_{1}(x) \bigg[ 1+\int_{0}^{t}\exp\{{\rho_{1} W_{r}}\}\nonumber\\&\qquad\times \left( \left\| \exp\{{\lambda r}\} S_{r}f_{1} \right\|_{\infty}^{\beta_{1}}+\left\| \exp\{{\lambda r}\} S_{r}f_{1} \right\|_{\infty}^{\gamma_{2}}\right) \mathscr{G}^{\beta_{1}+\gamma_{2}+2}_{1}(r)dr \bigg] \nonumber\\
				&=\exp\{{\lambda t}\} S_{t}f_{1}(x)\mathscr{G}_{1}(t).\nonumber
			\end{align}
			Similarly, we get 
			\begin{eqnarray}
				\exp\{{\lambda t}\}S_{t}f_{2}(x)\leq \mathcal{T}_{2}\left[ u,w_{2} \right](t,x)\leq \exp\{{\lambda t}\}S_{t}f_{2}(x) \mathscr{G}_{2}(t), \nonumber
			\end{eqnarray}
			for all $u, w_2$ such that $$0\leq u(t,x)\leq \exp\{{\lambda t}\} S_{t}f_{2}(x)\mathscr{G}_{2}(t), \ \ 0\leq w_{2}(t,x)\leq \exp\{{\lambda t}\} S_{t}f_{2}(x)\mathscr{G}_{2}(t).$$ Let us take, for $x\in D$ and $0\leq t\leq \tau_{*},$ 
			\begin{eqnarray}\label{in1}
			u^{(0)}_{1}(t,x)=\exp\{{\lambda t}\}S_{t}f_{1}(x), \ \ u^{(0)}_{2}(t,x)=\frac{L_{1}}{L_{2}} \exp\{{\lambda t}\}S_{t}f_{2}(x),  
			\end{eqnarray}  
			and
			\begin{eqnarray}\label{in2}
			u^{(n)}_{1}(t,x)=\mathcal{T}_{1}[ u^{(n-1)}_{2},u_{1}^{(n-1)} ](t,x), \ \ u^{(n)}_{2}(t,x)=\mathcal{T}_{2}[ u^{(n-1)}_{1},u_{2}^{(n-1)} ](t,x), \ \ n\geq1. 
			\end{eqnarray}
			Our aim is to show that the sequences of functions $\{ u^{(n)}_{1} \} \ \mbox{and} \ \{u^{(n)}_{2}\}$ are increasing.
			Let us consider
			\begin{align}
				u^{(0)}_{1}(t,x)&\leq \exp\{{\lambda t}\} S_{t}f_{1}(x)+\int_{0}^{t} \exp\{{\rho_{1} W_{r}+\lambda (t-r)}\} \left( S_{t-r}u^{(0)}_{2}(r,x) \right)^{1+\beta_{1}}dr\nonumber\\
				&\quad+\int_{0}^{t} \exp\{{\rho_{1} W_{r}+\lambda (t-r)}\} \left( S_{t-r}u_{1}^{(0)}(r,x) \right)^{1+\gamma_{2}}dr\nonumber\\
				&= \mathcal{T}_{1}[ u^{(0)}_{2},u_{1}^{(0)}](t,x) = u^{(1)}_{1}(t,x).\nonumber
			\end{align}
			Now assume that $u^{(n)}_{i}\geq u^{(n-1)}_{i}, \ i=1,2.$  Then 
			\begin{eqnarray}
				u^{(n+1)}_{1}=\mathcal{T}_{1} [ u^{(n)}_{2},u_{1}^{(n)}]\geq \mathcal{T}_{1}[ u^{(n-1)}_{2},u_{1}^{(n-1)}]=u^{(n)}_{1}.\nonumber
			\end{eqnarray}
			Similarly, we have
			\begin{eqnarray}
			u^{(n+1)}_{2}=\mathcal{T}_{2} [ u^{(n)}_{1},u_{2}^{(n)}]\geq \mathcal{T}_{2}[ u^{(n-1)}_{1},u_{2}^{(n-1)}]=u^{(n)}_{2},\nonumber
			\end{eqnarray}
			where we have used the monotonicity of the operators $\mathcal{T}_{1}$ and $\mathcal{T}_{2}$ to obtain the above inequality. Therefore, the limits 
			\begin{eqnarray}
				v_{1}(t,x)=\lim_{n\rightarrow \infty}u^{(n)}_{1}(t,x), \ \ v_{2}(t,x)=\lim_{n\rightarrow \infty}u^{(n)}_{2}(t,x), \nonumber
			\end{eqnarray}
			exist for $x\in D$ and $0 \leq t< \tau_{*}.$ Then by the monotone convergence theorem, we obtain 
			\begin{eqnarray}
				v_{1}(t,x)=\mathcal{T}_{1}\left[ v_{2},v_{1} \right](t,x), \  v_{2}(t,x)=\mathcal{T}_{2}\left[ v_{1},v_{2} \right](t,x), \ \ x\in D, \ \ 0 \leq t<\tau_{*}.\nonumber
			\end{eqnarray}
			Since $f_{1}=L_{1} \psi$ and $f_{2}=L_{2} \psi,$ for some positive constants $L_{1}$ and $L_{2}$ with $L_{1} \leq L_{2}$, from \eqref{in1} it follows that
			\begin{align}\label{eq2}
			\left\{
			\begin{aligned}
			0 &\leq u_{1}^{(0)}(t,x)=e^{\lambda t}S_{t}\left(\frac{L_{1}}{L_{2}} f_{2}(x)\right)=\frac{L_{1}}{L_{2}}e^{\lambda t}S_{t} f_{2}(x)\leq e^{\lambda t}S_{t} f_{2}(x) \mathscr{G}_{2}(t), \\
			0 &\leq \frac{L_{2}}{L_{1}} u_{2}^{(0)}(t,x)=e^{\lambda t}S_{t}\left(\frac{L_{2}}{L_{1}} f_{1}(x)\right)=\frac{L_{2}}{L_{1}}e^{\lambda t}S_{t} f_{1}(x), \\
			&u_{2}^{(0)}(t,x)=e^{\lambda t}S_{t} f_{1}(x)\leq e^{\lambda t}S_{t} f_{1}(x) \mathscr{G}_{1}(t).  
			\end{aligned}
			\right.   
			\end{align}
		From \eqref{in2}, we have
			\begin{align}\label{eq3}
			u^{(1)}_{1}(t,x)&=\mathcal{T}_{1}[u^{(0)}_{2},u_{1}^{(0)} ](t,x), 
			\end{align}
			Using  \eqref{in1} and \eqref{eq2}, we get 
			$u_{2}^{(0)}(t,x) \leq e^{\lambda t}S_{t} f_{1}(x)\mathscr{G}_{1}(t) \ \ \mbox{and} \ \  u_{1}^{(0)}(t,x) \leq e^{\lambda t}S_{t}f_{1}(x) \mathscr{G}_{1}(t).$  	
			Therefore from \eqref{eq3}, we obtain 
			\begin{align}\label{2}
			e^{\lambda t}S_{t}f_{1}(x) \leq u^{(1)}_{1}(t,x) \leq  e^{\lambda t}S_{t}f_{1}(x) \mathscr{G}_{1}(t).
			\end{align}
			From \eqref{in2}, we find 
			\begin{align}\label{b11}
			u^{(1)}_{2}(t,x)=\mathcal{T}_{2}[u^{(0)}_{1},u_{2}^{(0)} ](t,x). 
			\end{align}
			We know from \eqref{in1} and \eqref{eq2}  that $u_{1}^{(0)}(t,x) \leq e^{\lambda t}S_{t}f_{2}(x)\mathscr{G}_{2}(t),$ and $u_{2}^{(0)}(t,x) \leq e^{\lambda t}S_{t}f_{2}(x)\mathscr{G}_{2}(t)$.		
			Therefore from \eqref{b11}, we deduce 
			\begin{align}
			u^{(1)}_{2}(t,x)=\mathcal{T}_{2}[u^{(0)}_{1},u_{2}^{(0)} ](t,x) &\leq e^{\lambda t}S_{t}f_{2}(x) \mathscr{G}_{2}(t), \nonumber\\
			e^{\lambda t}S_{t}f_{2}(x) \leq u^{(1)}_{2}(t,x) \leq  e^{\lambda t}S_{t}f_{2}(x) \mathscr{G}_{2}(t). \nonumber 
			\end{align}
			By \eqref{in2}, we have
			\begin{align}\label{3}
			u^{(2)}_{1}(t,x)&=\mathcal{T}_{1}[u^{(1)}_{2},u_{1}^{(1)} ](t,x), 
			\end{align}
			where $u^{(1)}_{2}(t,x)=\mathcal{T}_{2}[u^{(0)}_{1},u_{2}^{(0)} ](t,x)$.
			We know that $u_{2}^{(0)}(t,x) \leq e^{\lambda t}S_{t} f_{1}(x)\mathscr{G}_{1}(t)$ and $ u_{1}^{(0)}(t,x) \leq e^{\lambda t}S_{t}f_{1}(x) \mathscr{G}_{1}(t)$. 
			Therefore, we infer 
			\begin{align}\label{9}
			u^{(1)}_{2}(t,x) \leq e^{\lambda t}S_{t} f_{1}(x)\mathscr{G}_{1}(t).
			\end{align}
			By using \eqref{9} and \eqref{2} in \eqref{3}, we have
			\begin{align}\label{4}
			e^{\lambda t}S_{t} f_{1}(x)\leq u^{(2)}_{1}(t,x) \leq e^{\lambda t}S_{t} f_{1}(x)\mathscr{G}_{1}(t).
			\end{align}
			From \eqref{in2}, we get 
			\begin{align}\label{5}
			u^{(2)}_{2}(t,x)=\mathcal{T}_{2}[u^{(1)}_{1},u_{2}^{(1)} ](t,x), 
			\end{align}
			where $u^{(1)}_{1}(t,x)=\mathcal{T}_{1}[u^{(0)}_{1},u_{2}^{(0)} ](t,x)$.		
			We know that $u^{(0)}_{1}(t,x) \leq e^{\lambda t}S_{t} f_{2}(x)\mathscr{G}_{2}(t)$ and $u_{2}^{(0)}(t,x) \leq e^{\lambda t}S_{t}f_{2}(x) \mathscr{G}_{2}(t)$.		
			Therefore, we arrive at  $0 \leq u^{(1)}_{1}(t,x) \leq e^{\lambda t}S_{t}f_{2}(x) \mathscr{G}_{2}(t).$ 	
			From \eqref{5}, we obtain
			\begin{align}\label{6}
			e^{\lambda t}S_{t} f_{1}(x)\leq u^{(2)}_{2}(t,x) \leq e^{\lambda t}S_{t} f_{2}(x)\mathscr{G}_{2}(t).
			\end{align}
			Continuing like this,  for each $n \geq 1$, one can deduce that 
			\begin{align}
			e^{\lambda t}S_{t} f_{1}(x)&\leq u^{(n)}_{1}(t,x) \leq e^{\lambda t}S_{t} f_{1}(x)\mathscr{G}_{1}(t), \nonumber\\
			e^{\lambda t}S_{t} f_{2}(x)&\leq u^{(n)}_{2}(t,x) \leq e^{\lambda t}S_{t} f_{2}(x)\mathscr{G}_{2}(t). \nonumber
			\end{align}
			Moreover,  we have
			\begin{align} 
				v_{1}(t,x)&= \mathcal{T}_{1}\left[v_2,v_1 \right](t,x)\leq \exp\{{\lambda t}\}S_{t}f_{1}(x) \mathscr{G}_{1}(t)\ \text{ and }\nonumber\\ v_{2}(t,x)&=\mathcal{T}_{2}\left[ v_1,v_2 \right](t,x)\leq \exp\{{\lambda t}\}S_{t}f_{2}(x) \mathscr{G}_{2}(t), \nonumber
			\end{align}
			so that
			\begin{align}
				v_{1}(t,x) &\leq \frac{e^{\lambda t}S_{t}f_{1}(x)}{\left[ 1-\left(\beta_{1}+\gamma_{2}+1\right) \int_{0}^{t} e^{{\rho_{1} W_{r}}} \left( \left\| e^{{\lambda r}}S_{r}f_{1} \right\|_{\infty}^{\beta_{1}}+\left\| e^{\lambda r}S_{r}f_{1} \right\|_{\infty}^{\gamma_{2}}\right) dr\right]^{\frac{1}{\beta_{1}+\gamma_{2}+1}}}, \nonumber\\
				v_{2}(t,x) &\leq \frac{e^{\lambda t}S_{t}f_{2}(x)}{\left[ 1-\left(\gamma_{1}+\beta_{2}+1\right) \int_{0}^{t} e^{\rho_{2} W_{r}} \left( \left\| e^{\lambda r}S_{r}f_{2} \right\|_{\infty}^{\gamma_{1}}+\left\| e^{\lambda r}S_{r}f_{2} \right\|_{\infty}^{\beta_{2}}\right)dr\right]^{\frac{1}{\gamma_{1}+\beta_{2}+1}}}. \nonumber
			\end{align}
			By the choice of initial values as in (\ref{q1}), the proof of the theorem follows.
		\end{proof}
	\end{theorem}
	
	\begin{Rem} 
		For  general bounded, measurable and positive $f_{i}, \ i=1,2,$ one can show that the blow-up time of \eqref{b1}-\eqref{b2} is bounded below by the random time 
		\begin{align}
			\inf \Bigg\{ &t\geq 0 : \int_{0}^{t} \exp\{{\rho_{1}W_{r}}\}\left( \left\| \exp\{{\lambda r}\}S_{r}f_{1} \right\|_{\infty}^{\beta_{1}}+\left\| \exp\{{\lambda r}\}S_{r}f_{1} \right\|_{\infty}^{\gamma_{2}} \right)dr \geq \left(\beta_{1}+\gamma_{2}+1 \right)^{-1} \nonumber\\
			&\mbox{ or }\int_{0}^{t} \exp\{{\rho_{2}W_{r}}\} \left( \left\| \exp\{{\lambda r}\}S_{r}f_{2} \right\|_{\infty}^{\gamma_{1}}+\left\| \exp\{{\lambda r}\}S_{r}f_{2} \right\|_{\infty}^{\beta_{2}} \right)dr \geq \left(\gamma_{1}+\beta_{2}+1 \right)^{-1} \Bigg\}, \nonumber 
		\end{align} 
		which coincides with $\tau_{\ast},$ when the initial values satisfy (\ref{q1}).
	\end{Rem}
		
	\section{Upper Bound}\label{sec4}
	In this section, we obtain the upper bound $\tau^*$ for the blow-up time $\tau$, under the assumption \eqref{a4} and for any initial values $f_{i} \geq 0, \ i=1,2$. By setting $\varphi_{i}=\psi$, the system (\ref{c3}) reduces to
	\begin{align} \label{a5}
		v_{i}(t,\psi) &= v_{i}(0,\psi) + \int_{0}^{t} \left[ v_{i}(s,(\Delta + V_{i})\psi) - \frac{k_{i}^{2}}{2} v_{i}(s,\psi) \right] ds \nonumber\\
		&\quad+ \int_{0}^{t} \exp\{{-k_{i} W_{s}}\} (\exp\{{k_{i} W_{s}}\} v_{i})^{1+\beta_{i}}(s,\psi) ds\nonumber\\
		&\quad+ \int_{0}^{t} \exp\{{-k_{i} W_{s}}\} (\exp\{{k_{j} W_{s}}\} v_{j})^{1+\gamma_{j}}(s,\psi) ds. 
	\end{align}
	Next, we consider the integral $\displaystyle\int_{0}^{t} \left[ v_{i}(s,(\Delta + V_{i})\psi) - \frac{k_{i}^{2}}{2} v_{i}(s,\psi) \right] ds,$ and using (\ref{a3}) and (\ref{a4}), we have
	\begin{align}
		&	\int_{0}^{t} \left[ v_{i}(s,(\Delta + V_{i})\psi) - \frac{k_{i}^{2}}{2} v_{i}(s,\psi) \right] ds \nonumber\\&= \int_{0}^{t} \int_{D} \left[v_{i} (s,x)(\Delta+V_{i})\psi(x) -\frac{k_{i}^{2}}{2}v_{i}(s,x)\psi(x)\right]dxds \nonumber\\
		&=\int_{0}^{t} \int_{D} v_{i} (s,x) \Big[\Delta+\lambda \Big] \psi(x)dxds 
		= 0. \nonumber
	\end{align}
	Therefore, (\ref{a5}) becomes
	\begin{align}
		\frac{dv_{i}(t,\psi)}{dt} &= \exp{\{-k_{i} W_{t}\}}(\exp\{k_{i} W_{t}\}v_{i})^{1+\beta_{i}} (t,\psi)+\exp\{-k_{i} W_{t}\}(\exp\{k_{j} W_{t}\} v_{j})^{1+\gamma_{j}}(t,\psi), \nonumber
	\end{align}
	for $i=1,2, \ j\in \left\{ 1,2 \right\} / \left\{i\right\}. $
	By using Jensen's inequality, we obtain 
	\begin{align*}
	(\exp\{{k_{i} W_{s}}\} v_{i})^{1+\beta_{i}}(s,\psi) &\geq \left[ \int_{D} (\exp\{ k_{i} W_{s}\}v_{i}(s,x))\psi(x)dx\right]^{1+\beta_{i}} \\&=\exp\{(1+\beta_{i})k_{i}W_{s}\}v_{i}(s,\psi)^{1+\beta_{i}},\nonumber
	\end{align*}
	and  
	\begin{eqnarray}
	(\exp\{{k_{j} W_{s}}\} v_{j})^{1+\gamma_{j}}(s,\psi) \geq \exp\{(1+\gamma_{j})k_{j}W_{s}\}v_{j}(s,\psi)^{1+\gamma_{j}}.\nonumber
	\end{eqnarray}
	Thus, it is immediate that  for $ i=1,2, \ j\in \left\{ 1,2 \right\} / \left\{i\right\}$
	\begin{align}
	\frac{dv_{i}(t,\psi)}{dt} &\geq e^{-k_{i} W_{t}+(1+\beta_{i})k_{i} W_{t}} {v_{i} (t,\psi)}^{1+\beta_{i}}+e^{-k_{i} W_{t}+(1+\gamma_{j})k_{j} W_{t}} {v_{j} (t,\psi)}^{1+\gamma_{j}}. \nonumber
	\end{align}
	In this way, $v_{i}(t,\psi) \geq h_{i}(t),$ for $ i=1,2,$ where
	\begin{equation} \label{ab2}
	\left\{
	\begin{aligned}
	\frac{dh_{1}(t)}{dt}&=\exp\{{\rho_{1} W_{t}}\} \left[ h_{1}^{1+\beta_{1}}(t)+h_{2}^{1+\gamma_{2}}(t) \right], \\ \frac{dh_{2}(t)}{dt}&=\exp\{{\rho_{2} W_{t}}\} \left[ h_{1}^{1+\gamma_{1}}(t)+h_{2}^{1+\beta_{2}}(t) \right],\\
	h_{i}(0)&=v_{i}(0,\psi), \ i=1,2.  
	\end{aligned}
	\right.
	\end{equation}
	Let  us define $E(t):=h_{1}(t)+h_{2}(t), \ t\geq 0,$ so that $E(\cdot)$ satisfies 
	\begin{equation} \label{a6}
			\left\{
		\begin{aligned}
		\frac{dE(t)}{dt}&=\exp\{{\rho_{1} W_{t}}\} \left[ h_{1}^{1+\beta_{1}}(t)+h_{2}^{1+\gamma_{2}}(t) \right] + \exp\{{\rho_{2} W_{t}}\} \left[ h_{1}^{1+\gamma_{1}}(t)+h_{2}^{1+\beta_{2}}(t) \right], \\
		E(0)&=\int_D(f_1(x)+f_2(x))\psi (x)dx. 
		\end{aligned}
		\right. 
	\end{equation}
	\begin{theorem} \label{thm4.1} 
		\indent \\
		1. Assume that $\beta_{1}=\gamma_{2}=\gamma_{1}=\beta_{2}=m
		\ (say)>0$. Then for  $\rho=\rho_{1}=\rho_{2}$ (see\eqref{a4}),  $\tau \leq \tau^{\ast}$, where
		\begin{eqnarray}
		\tau^{\ast}=\inf \left\lbrace t\geq 0 : \int_{0}^{t} \exp\{{\rho W_{s}}\} ds \geq 2^{m}{m}^{-1}E^{-m}(0) \right\rbrace .\nonumber
		\end{eqnarray}  
		2. For $\beta_{1}=\gamma_{1}=\beta, \ \gamma_{2}=\beta_{2}=\gamma$ with $\beta>\gamma>0$, let $D_{1}=\left( \frac{\beta-\gamma}{1+\beta} \right)\left(\frac{1+\beta}{1+\gamma} \right)^{\frac{1+\gamma}{\beta-\gamma}}$ and
		\begin{align}\label{nca3}
		\epsilon_{0}&\leq \min \Bigg\{ 1, \left( h_{1}(0)/D_{1}^{1/1+\gamma} \right)^{\beta-\gamma} \Bigg\}.
		\end{align} 
		Assume that 
		\begin{align}\label{nca4}
		2^{-(1+\gamma)}\epsilon_{0}E^{1+\gamma}(0)\geq \epsilon_{0}^{\frac{1+\beta}{\beta-\gamma}}D_{1}.
		\end{align}
		Then $\tau\leq\tau^{\ast}$, where 
		\begin{eqnarray}
		\tau^{\ast} = \inf \left\lbrace  t\geq 0 : \int_{0}^{t} \left( \exp\{\rho_{1} W_{s}\} \wedge \exp\{\rho_{2} W_{s}\} \right)  ds \geq {\left[\gamma E^{\gamma}(0)\left(\frac{\epsilon_0}{2^{1+\gamma}}-\frac{ \epsilon_{0}^{\frac{1+\beta}{\beta-\gamma}}D_{1}}{E^{1+\gamma}(0)}\right)\right]^{-1}}\right\rbrace, \nonumber
		\end{eqnarray}
	and  $\rho_1,\rho_2$ are defined in \eqref{a4}.\\
		3. Let $\beta_{1}>\gamma_{2}>\gamma_{1}>\beta_{2}>0$ and let $D_{2}=\frac{\beta_{1}-\gamma_{2}}{1+\beta_{1}}\left( \frac{1+\beta_{1}}{1+\gamma_{2}} \right)^{\frac{1+\gamma_{2}}{\beta_{1}-\gamma_{2}}}, \ D_{3}=\frac{\gamma_{1}-\beta_{2}}{1+\gamma_{1}}\left( \frac{1+\gamma_{1}}{1+\beta_{2}} \right)^{\frac{1+\beta_{2}}{\gamma_{1}-\beta_{2}}},\\  D_{4}=\frac{\gamma_{2}-\beta_{2}}{1+\gamma_{2}}\left( \frac{1+\gamma_{2}}{1+\beta_{2}} \right)^{\frac{1+\beta_{2}}{\gamma_{2}-\beta_{2}}}$ 
		and 
		\begin{align}\label{4.4}
		\epsilon_{0}&\leq\min \Bigg\{ 1, \left( h_{1}(0)/ D_{2}^{1/1+\gamma_{2}} \right)^{\beta_{1}-\gamma_{2}}, \left( h_{1}(0)/D_{3}^{1/1+\beta_{2}}\right)^{\gamma_{1}-\beta_{2}}\Bigg\}.
		\end{align} 
		Assume that 
		\begin{eqnarray} \label{tm2}
		2^{-(1+\gamma_{2})}\epsilon_{0}E^{1+\beta_{2}}(0)\geq \epsilon_{0}^{\frac{1+\beta_{1}}{\beta_{1}-\gamma_{2}}} D_{2} + \epsilon_{0}^{\frac{1+\gamma_{1}}{\gamma_{1}-\beta_{2}}}D_{3}+\epsilon_{0}^{\frac{ 1+\gamma_{2}}{\gamma_{2}-\beta_{2}}}D_{4}.
		\end{eqnarray}
		Then $\tau\leq\tau^{\ast}$, where 
		\begin{align*}
		\tau^{\ast} = \inf \Bigg\{  t\geq 0 : \int_{0}^{t} &\left( \exp\{\rho_{1} W_{s}\} \wedge \exp\{\rho_{2} W_{s}\} \right)  ds \geq \\
		&\left. \left(\beta_{2}E^{\beta_{2}}(0)\left[ \frac{\epsilon_{0}}{2^{(1+\gamma_{2})}} -\frac{\left(  \epsilon_{0}^{\frac{1+\beta_{1}}{\beta_{1}-\gamma_{2}}} D_{2} + \epsilon_{0}^{\frac{1+\gamma_{1}}{\gamma_{1}-\beta_{2}}}D_{3}+\epsilon_{0}^{\frac{ 1+\gamma_{2}}{\gamma_{2}-\beta_{2}}}D_{4}\right)}{E^{1+\beta_{2}}(0)}\right]\right)^{-1}\right\}, \nonumber
		\end{align*}
		and $\rho_1,\rho_2$ are defined in \eqref{a4}.	
			
		\begin{Rem}
		\textcolor{blue}{Note that \eqref{nca4} and (\ref{tm2}) follow from the conditions 
		\begin{eqnarray}
		\epsilon_{0}^{\frac{1+\gamma}{\beta-\gamma}}\leq E^{1+\gamma}(0)2^{-(1+\gamma)}D_{1}^{-1}, \nonumber 
		\end{eqnarray} and		
		\begin{eqnarray} \label{RE1}
		\epsilon_{0}^{\min\{ M,N \}} \leq (D_{2}+D_{3}+D_{4})^{-1} 2^{-(1+\gamma_{2})}E^{1+\beta_{2}}(0),
		\end{eqnarray}
		where $M= \frac{1+\gamma_{2}}{\beta_{1}-\gamma_{2}}, N= \frac{1+\beta_{2}}{\gamma_{2}-\beta_{2}}$. Using the fact that
		$\epsilon_0\geq\epsilon_0^{a_1}A_1+ \epsilon_0^{a_2}A_2+\epsilon_0^{a_3}A_3$ 
		is satisfied provided
		$\epsilon_0\geq\epsilon_0^{\min\{a_1,a_2,a_3\}}(A_1+A_2+A_3),$ since $$\epsilon_0^{\min\{a_1,a_2,a_3\}}(A_1+A_2+A_3)\geq \epsilon_0^{a_1}A_1+ \epsilon_0^{a_2}A_2+\epsilon_0^{a_3}A_3,$$ for each positive constants $A_1,A_2,A_3,a_1,a_2,a_3$ and $0<\epsilon_{0}<1$.
		The inequality \eqref{RE1} is readily obtained by replacing $A_1,A_2,A_3,a_1, a_2,a_3$ by $D_2,D_3,D_4, M,N,N,$ respectively in the last inequality, so that we deduce 
		$$2^{-(1+\gamma_{2})}E^{1+\beta_{2}}(0)\geq \epsilon_{0}^{\frac{1+\gamma_{2}}{\beta_{1}-\gamma_{2}}} D_{2} +(D_{3}+D_{4}) \epsilon_{0}^{\frac{1+\beta_{2}}{\gamma_{2}-\beta_{2}}}.$$
		Note that $\gamma_{2}>\gamma_{1}>\beta_{2}>0$ and hence the above inequality becomes
		$$2^{-(1+\gamma_{2})}E^{1+\beta_{2}}(0)\geq \epsilon_{0}^{\frac{1+\gamma_{2}}{\beta_{1}-\gamma_{2}}} D_{2} + \epsilon_{0}^{\frac{1+\beta_{2}}{\gamma_{1}-\beta_{2}}}D_{3}+\epsilon_{0}^{\frac{ 1+\beta_{2}}{\gamma_{2}-\beta_{2}}}D_{4}.$$
		By multiplying both sides of the above inequality by $\epsilon_0$, we arrive at (\ref{tm2}).} One can refer to Theorem 3.1 in \cite{chow11} and \cite{doz2013} for similar conditions.
		\end{Rem}

		\begin{proof}[Proof of Theorem \ref{thm4.1}] 
		\textbf{Case 1:} Suppose that if $\beta_{1}=\gamma_{1}=\gamma_{2}=\beta_{2}=m \ (say)>0$ in (\ref{a4}). Then, we get, $\rho_{1}=\rho_{2}=\rho \ (\mbox{say})$ and 
		\begin{eqnarray} \label{a7}
		\frac{dE(t)}{dt}=2\exp\{{\rho W_{t}}\} \left[ h_{1}^{1+m}(t)+h_{2}^{1+m}(t) \right].
		\end{eqnarray}
		By substituting $a=1, \ b=\frac{h_{1}}{h_{2}}\ \mbox{and}\ n_{1}=1+m$ into the inequality 
		\begin{eqnarray}\label{cta1}
		a^{n_{1}}+b^{n_{1}}\geq 2^{-n_{1}}{(a+b)}^{n_{1}},
		\end{eqnarray}
		which is valid for $a, \ b\geq0,$ we obtain 
		\begin{eqnarray}
		&&h_{1}^{1+m}+h_{2}^{1+m} \geq 2^{-(1+m)} \left( h_{1}+h_{2} \right)^{1+m} \geq 2^{-(1+m)} E^{1+m}(t).\nonumber
		\end{eqnarray}
		From $\left( \ref{a7} \right) $, we infer that 
		\begin{align*}
		\frac{dE(t)}{dt}\geq 2^{-m} \exp\{{\rho W_{t}}\} E^{1+m}(t).\nonumber
		\end{align*} 
		Thus, $E(t)$ blows up not later than the solution $I(t)$ of the equation
		\begin{align}
		\frac{dI(t)}{dt}= 2^{-m} \exp\{{\rho W_{t}}\} I(t)^{1+m}, \ I(0)=E(0),\nonumber
		\end{align}
		which immediately  gives 
		\begin{align}
		I(t)^{-m}=I(0)^{-m}-m2^{-m}\int_{0}^{t} \exp\{{\rho W_{s}}\} ds. \nonumber
		\end{align} 
		For the above equation, the blow-up time is given by
		\begin{align*}
		\tau^{\ast} = \inf \left\lbrace t\geq 0 : \int_{0}^{t} \exp\{{\rho W_{s}}\} ds\geq 2^{m} {m}^{-1}E^{-m}(0) \right\rbrace. 
		\end{align*}
	Since $I(\cdot)\leq E(\cdot)\leq v_{1}(\cdot,\psi)+v_{2}(\cdot,\psi)$, $\tau^*$ is an upper bound for the blow-up time for $v(\cdot,\psi)=(v_1(\cdot,\psi),v_2(\cdot,\psi))^{\top}$. 
		
		\vskip 0.2 cm
		\noindent 
		\textbf{Case 2:} Suppose $\beta_{1}=\gamma_{1}=\beta, \ \beta_{2}=\gamma_{2}=\gamma$ with $\beta>\gamma>0,$ then we have 
		\begin{align} \label{nca1}
		\frac{dE(t)}{dt}&=\exp\{{\rho_{1} W_{t}}\} \left[ h_{1}^{1+\beta}(t)+h_{2}^{1+\gamma}(t) \right]+ \exp\{{\rho_{2} W_{t}}\} \left[ h_{1}^{1+\beta}(t)+h_{2}^{1+\gamma}(t) \right]\nonumber\\
		&\geq \left( \exp\{\rho_{1} W_{t}\} \wedge \exp\{\rho_{2} W_{t}\} \right)   \left[ h_{1}^{1+\beta}(t)+h_{2}^{1+\gamma}(t) \right].
		\end{align}
		The Young inequality states that if $1<b<\infty$, $\delta>0$ and $a=\frac{b}{b-1}$, then 
		\begin{eqnarray} \label{na10}
		xy\leq \frac{\delta^{a}x^{a}}{a}+\frac{\delta^{-b}y^{b}}{b}, \ x,y\geq 0. 
		\end{eqnarray}
		By setting $b=\frac{1+\beta}{1+\gamma}, \ y=h_{1}^{1+\gamma}(t), \ x=\epsilon, \ \delta=\left( \frac{1+\beta}{1+\gamma} \right)^{\frac{1+\gamma}{1+\beta}}$ and using the fact that $\gamma<\beta$ in (\ref{na10}), it follows that for any $\epsilon>0,$
		\begin{eqnarray} 
		h_{1}^{1+\beta}(t) \geq \epsilon h_{1}^{1+\gamma}(t)-D_{1}\epsilon^{\frac{1+\beta}{\beta-\gamma}}\nonumber.
		\end{eqnarray}
		Since $\epsilon_0$ is the minimum of the quantities given in \eqref{nca3}, in particular, we have $\epsilon_0\leq ( h_{1}(0)/D_{1}^{1/1+\gamma})^{\beta-\gamma},$ so that 
		\begin{eqnarray}
		\epsilon_{0} h_{1}^{1+\gamma}(0)-D_{1}\epsilon_{0}^{\frac{1+\beta}{\beta-\gamma}}\geq 0\nonumber.
		\end{eqnarray}
		From $\eqref{nca1},$ we have
		\begin{align}\label{nca2}
		\frac{dE(t)}{dt}\geq \left( \exp\{\rho_{1} W_{t}\} \wedge \exp\{\rho_{2} W_{t}\} \right)  \left[ \epsilon_{0} h_{1}^{1+\gamma}(t)+h_{2}^{1+\gamma}(t)-D_{1}\epsilon_{0}^{\frac{1+\beta}{\beta-\gamma}} \right].
		\end{align}
		Utilizing the inequality \eqref{cta1} again with $n_{1}=1+\gamma, \ a=1$ and $b=\epsilon_{0}^\frac{1}{1+\gamma}\frac{h_{1}(t)}{h_{2}(t)}$, and by using \eqref{nca3}, we see that 
		\begin{eqnarray}
		h_{2}^{1+\gamma}(t) + \epsilon_{0} h_{1}^{1+\gamma}(t) \geq 2^{-(1+\gamma)} \left( h_{2}(t)+\epsilon_{0}^\frac{1}{1+\gamma}h_{1}(t) \right)^{1+\gamma}\geq  2^{-(1+\gamma)}\epsilon_{0}E^{1+\gamma}(t)\nonumber. 
		\end{eqnarray}
		From $\eqref{nca2},$ we deduce that
		\begin{align}
		\frac{dE(t)}{dt}\geq \left( \exp\{\rho_{1} W_{t}\} \wedge \exp\{\rho_{2} W_{t}\} \right)  \left[2^{-(1+\gamma)}\epsilon_{0}E^{1+\gamma}(t)-D_{1}\epsilon_{0}^{\frac{1+\beta}{\beta-\gamma}} \right].\nonumber
		\end{align}
		Note that the condition \eqref{nca4} gives $E(t)\geq E(0)>0$ and so
		\begin{align}
		\frac{dE(t)}{E^{1+\gamma}(t)}\geq \left( \exp\{\rho_{1} W_{t}\} \wedge \exp\{\rho_{2} W_{t}\} \right)  \left[\frac{\epsilon_{0}}{2^{1+\gamma}}-\frac{D_{1}\epsilon_{0}^{\frac{1+\beta}{\beta-\gamma}}}{E^{1+\gamma}(0)} \right]dt.\nonumber
		\end{align}
		Thus, $E(t)$ blows up not later than the solution $I(t)$ of the equation
		\begin{align}
		\frac{dI(t)}{I^{1+\gamma}(t)}=\left( \exp\{\rho_{1} W_{t}\} \wedge \exp\{\rho_{2} W_{t}\} \right)  \left[\frac{\epsilon_{0}}{2^{1+\gamma}}-\frac{D_{1}\epsilon_{0}^{\frac{1+\beta}{\beta-\gamma}}}{I^{1+\gamma}(0)} \right]dt, \ I(0)=E(0),\nonumber
		\end{align}
		and
		\begin{align*}
		I(t)= \left\lbrace E^{-\gamma}(0)-\gamma\left[ \frac{\epsilon_0}{2^{1+\gamma}}-\frac{ \epsilon_{0}^{\frac{1+\beta}{\beta-\gamma}}D_{1} }{E^{1+\gamma}(0)} \right] \displaystyle\int_{0}^{t} \left( \exp\{\rho_{1} W_{s}\} \wedge \exp\{\rho_{2} W_{s}\} \right)  ds \right\rbrace^{\frac{1}{-\gamma}}.
		\end{align*}
		For the above inequality, the blow-up time is given by
		\begin{eqnarray}
		\tau^{\ast} = \inf \left\lbrace  t\geq 0 : \int_{0}^{t} \left( \exp\{\rho_{1} W_{s}\} \wedge \exp\{\rho_{2} W_{s}\} \right) ds \geq {\left[\gamma E^{\gamma}(0)\left(\frac{\epsilon_0}{2^{1+\gamma}}-\frac{ \epsilon_{0}^{\frac{1+\beta}{\beta-\gamma}}D_{1}}{E^{1+\gamma}(0)}\right)\right]^{-1}}\right\rbrace. \nonumber
		\end{eqnarray}

			\vskip 0.2 cm
			\noindent 
			\textbf{Case 3:} Suppose if $\beta_{1}>\gamma_{2}>\gamma_{1}>\beta_{2}>0$, then
			by setting $b=\frac{1+\beta_{1}}{1+\gamma_{2}}, \ y=h_{1}^{1+\gamma_{2}}(t), \ x=\epsilon, \ \delta=\left( \frac{1+\beta_{1}}{1+\gamma_{2}} \right)^{\frac{1+\gamma_{2}}{1+\beta_{1}}}$ and using the fact that $\gamma_{2}<\beta_{1},$ in (\ref{na10}), it follows that for any $\epsilon>0,$
			\begin{eqnarray} \label{a11}
			h_{1}^{1+\beta_{1}}(t) \geq \epsilon h_{1}^{1+\gamma_{2}}(t)- D_{2}\epsilon^{\frac{1+\beta_{1}}{\beta_{1}-\gamma_{2}}}.
			\end{eqnarray}
			Since $\epsilon_0$ is the minimum of the quantities given in \eqref{4.4}, in particular, we have $\epsilon_0\leq ( h_{1}(0)/ D_{2}^{1/1+\gamma_{2}})^{\beta_{1}-\gamma_{2}},$ so that 
			\begin{eqnarray} \label{a12}
				\epsilon_{0} h_{1}^{1+\gamma_{2}}(0)- D_{2}\epsilon_{0}^{\frac{1+\beta_{1}}{\beta_{1}-\gamma_{2}}}\geq 0.
			\end{eqnarray}
			From (\ref{a11}), one can deduce that 
			\begin{eqnarray}
				h_{1}^{1+\beta_{1}}(t)+h_{2}^{1+\gamma_{2}}(t)\geq h_{2}^{1+\gamma_{2}}(t)+\epsilon_{0}h_{1}^{1+\gamma_{2}}(t)-\epsilon_{0}^{\frac{1+\beta_{1}}{\beta_{1}-\gamma_{2}}}D_{2}. \nonumber
			\end{eqnarray}
			Utilizing the inequality \eqref{cta1} again with $n_{1}=\gamma_{2}+1,a=1$ and $b=\epsilon_{0}^\frac{1}{1+\gamma_{2}}\frac{h_{1}(t)}{h_{2}(t)}$ and  using \eqref{4.4}, we see that 
			\begin{eqnarray}{\label{Ras1}}
				h_{2}^{1+\gamma_{2}}(t) + \epsilon_{0} h_{1}^{1+\gamma_{2}}(t) \geq 2^{-(1+\gamma_{2})} \left( h_{2}(t)+\epsilon_{0}^\frac{1}{1+\gamma_{2}}h_{1}(t) \right)^{1+\gamma_{2}}\geq  2^{-(1+\gamma_{2})}\epsilon_{0}E^{1+\gamma_{2}}(t). 
			\end{eqnarray}
			Therefore, it is immediate that 
			\begin{align*}
				h_{1}^{1+\beta_{1}}(t)+h_{2}^{1+\gamma_{2}}(t)\geq 2^{-(1+\gamma_{2})}\epsilon_{0}E^{1+\gamma_{2}}(t)-\epsilon_{0}^{\frac{1+\beta_{1}}{\beta_{1}-\gamma_{2}}} D_{2}. 
			\end{align*}
			Likewise, we derive
			\begin{eqnarray}
				h_{1}^{1+\gamma_{1}}(t)+h_{2}^{1+\beta_{2}}(t)\geq 2^{-(1+\beta_{2})}\epsilon_{0}E^{1+\beta_{2}}(t)-\epsilon_{0}^{\frac{1+\gamma_{1}}{\gamma_{1}-\beta_{2}}} D_{3}. \nonumber
			\end{eqnarray}
			Therefore (\ref{a6}) becomes 
			\begin{align*}
				\frac{dE(t)}{dt}&\geq \exp\{{\rho_{1} W_{t}}\} \left[ 2^{-(1+\gamma_{2})} \epsilon_{0}E^{1+\gamma_{2}}(t)-\epsilon_{0}^{\frac{1+\beta_{1}}{\beta_{1}-\gamma_{2}}}D_{2} \right]\nonumber\\&\quad+ \exp\{{\rho_{2} W_{t}}\} \left[ 2^{-(1+\beta_{2})} \epsilon_{0}E^{1+\beta_{2}}(t)-\epsilon_{0}^{\frac{1+\gamma_{1}}{\gamma_{1}-\beta_{2}}}D_{3} \right].
			\end{align*}
			Since, it is clear that 
			\begin{align*}
			\frac{dE(t)}{dt}&\geq \left( \exp\{\rho_{1} W_{t}\} \wedge \exp\{\rho_{2} W_{t}\} \right) \left[2^{-(1+\gamma_{2})} \epsilon_{0}\left(E^{1+\gamma_{2}}(t) +E^{1+\beta_{2}}(t) \right) -(\epsilon_{0}^{\frac{1+\beta_{1}}{\beta_{1}-\gamma_{2}}}D_{2} + \epsilon_{0}^{\frac{1+\gamma_{1}}{\gamma_{1}-\beta_{2}}}D_{3})\right].	
\end{align*}
	By setting $b=\frac{1+\gamma_{2}}{1+\beta_{2}}, \ y=E^{1+\beta_{2}}(t), \ x=\epsilon, \ \delta=\left( \frac{1+\gamma_{2}}{1+\beta_{2}} \right)^{\frac{1+\beta_{2}}{1+\gamma_{2}}}$ and using the fact that $\gamma_{2}>\beta_{2},$ in (\ref{na10}), it follows that for any $\epsilon>0,$	
	\begin{align}
	E^{1+\gamma_{2}}(t) \geq \epsilon E^{1+\beta_{2}}(t)-D_{4} \epsilon^{\frac{ 1+\gamma_{2}}{\gamma_{2}-\beta_{2}}}.
	\end{align}
		Since $\epsilon_0$ is the minimum of the quantities given in \eqref{tm2}, in particular, we have $\epsilon_0\leq ( E(0)/D_{4}^{1/1+\beta_{2}})^{\gamma_{2}-\beta_{2}}.$ Note that the condition \eqref{tm2} gives $E(t)\geq E(0)>0$, which implies 
		\begin{align*}
		\frac{dE(t)}{E^{1+\beta_{2}}(t)}&\geq \left( \exp\{\rho_{1} W_{t}\} \wedge \exp\{\rho_{2} W_{t}\} \right) \left[ \frac{\epsilon_{0}}{2^{(1+\gamma_{2})}} -\frac{\left(  \epsilon_{0}^{\frac{1+\beta_{1}}{\beta_{1}-\gamma_{2}}} D_{2} + \epsilon_{0}^{\frac{1+\gamma_{1}}{\gamma_{1}-\beta_{2}}}D_{3}+\epsilon_{0}^{\frac{ 1+\gamma_{2}}{\gamma_{2}-\beta_{2}}}D_{4}\right)}{E^{1+\beta_{2}}(0)}\right]dt,
		\end{align*}
		and 
			\begin{align*}
			E^{\beta_{2}}(t)\geq \frac{E^{\beta_{2}}(0)}{1-\beta_{2}E^{\beta_{2}}(0)\left[ \frac{\epsilon_{0}}{2^{(1+\gamma_{2})}} -\frac{\left(  \epsilon_{0}^{\frac{1+\beta_{1}}{\beta_{1}-\gamma_{2}}} D_{2} + \epsilon_{0}^{\frac{1+\gamma_{1}}{\gamma_{1}-\beta_{2}}}D_{3}+\epsilon_{0}^{\frac{ 1+\gamma_{2}}{\gamma_{2}-\beta_{2}}}D_{4}\right)}{E^{1+\beta_{2}}(0)}\right] \displaystyle\int_{0}^{t} \left( e^{\rho_{1} W_{s}} \wedge e^{\rho_{2} W_{s}} \right) ds}.
			\end{align*}
			For the above inequality, the blow-up time is given by
			\begin{align}
			\tau^{\ast} = \inf \Bigg\{  t\geq 0 : \int_{0}^{t} &\left( \exp\{\rho_{1} W_{s}\} \wedge \exp\{\rho_{2} W_{s}\} \right) ds \geq \nonumber\\
			&\left. {\left(\beta_{2}E^{\beta_{2}}(0)\left[ \frac{\epsilon_{0}}{2^{(1+\gamma_{2})}} -\frac{\left(  \epsilon_{0}^{\frac{1+\beta_{1}}{\beta_{1}-\gamma_{2}}} D_{2} + \epsilon_{0}^{\frac{1+\gamma_{1}}{\gamma_{1}-\beta_{2}}}D_{3}+\epsilon_{0}^{\frac{ 1+\gamma_{2}}{\gamma_{2}-\beta_{2}}}D_{4}\right)}{E^{1+\beta_{2}}(0)}\right]\right)^{-1}}\right\}, \nonumber 
			\end{align}		
		which completes the proof. 
		\end{proof}
	\end{theorem}

	\section{A More General Case}\label{sec5}
	We consider the system \eqref{b1}-\eqref{b2} with  the assumptions $\beta_{1},\beta_{2},\gamma_{1},\gamma_{2}>0$ and $V_{i}=\lambda+\frac{k^{2}_{i}}{2}, \ \ i=1,2.$ From (\ref{a2}), we have 
	\begin{align*}
		v_{i}(t,x)&=\exp\{{\lambda t}\} S_{t}f_{i}(x)+\int_{0}^{t} \exp\{{(-k_{i}+(1+\beta_{i})k_{i})W_{r}+\lambda (t-r)}\} S_{t-r}(v_{i}(r,.))^{1+\beta_{i}}(x) dr \nonumber\\
		&\quad+\int_{0}^{t} \exp\{{(-k_{i}+(1+\gamma_{j})k_{j})W_{r}+\lambda (t-r)}\} S_{t-r}(v_{j}(r,.))^{1+\gamma_{j}}(x) dr,  \nonumber
	\end{align*} for $ i=1,2, \ j\in \left\{ 1,2 \right\} / \left\{i\right\}.$
	\begin{theorem}\label{thm5.1}
		Assume that $\beta_{1},\beta_{2},\gamma_{1},\gamma_{2}>0$ and $V_{i}=\lambda+\frac{k^{2}_{i}}{2}, \ \ i=1,2,$ and let $f_{1}=L_{1}\psi$ and $f_{2}=L_{2}\psi,$ for some positive constant $L_{1}$ and $L_{2}$ with $L_{1}\leq L_{2}.$ Let $\tau_{\ast \ast}$ be given by 
		\begin{align}\label{e2}
			\tau_{**}= \inf \Bigg\{ t\geq 0 : \int^{t}_{0}  \left( e^{((1+\beta_{1})k_{1}-k_{1})W_{r}}\vee e^{((1+\gamma_{2})k_{2}-k_{1})W_{r}} \right) dr &\geq  \frac{1}{(\beta_{1}+\gamma_{2}+1)(L^{\beta_{1}}_{1}\left\|\psi \right\|_{\infty}^{\beta_{1}}+L^{\gamma_{2}}_{1}\left\|\psi \right\|_{\infty}^{\gamma_{2}})}\nonumber\\ \ \mbox{or} \ \int^{t}_{0} \left( e^{((1+\beta_{2})k_{2}-k_{2})W_{r}}\vee e^{((1+\gamma_{1})k_{1}-k_{2})W_{r}} \right) dr&\geq  \frac{1}{(\gamma_{1}+\beta_{2}+1)(L^{\gamma_{1}}_{2}\left\|\psi \right\|_{\infty}^{\gamma_{1}}+L^{\beta_{2}}_{2}\left\|\psi \right\|_{\infty}^{\beta_{2}})} 
			\Bigg\}. 
		\end{align}
		Then $\tau_{\ast \ast} \leq \tau.$ 
		\begin{proof}
			For $i=1,2$, we define
			\begin{align*}
				\mathcal{T}_{i}[v,w] &= \exp\{{\lambda t}\} S_{t}f_{i}(x)+\int_{0}^{t} \exp\{{((1+\beta_{i})k_{i}-k_{i})W_{r} +\lambda (t-r)}\} (S_{t-r} v(r,x))^{1+\beta_{i}} dr \\
				&\quad+\int_{0}^{t} \exp\{{((1+\gamma_{j})k_{j}-k_{i})W_{r} +\lambda (t-r)}\} (S_{t-r} w(r,x))^{1+\gamma_{j}} dr ,
			\end{align*}
			and
			\begin{align*}
			\mathscr{G}_{i}(t)= \Bigg[ 1-\left( \beta_{i}+\gamma_{j}+1 \right)& \int_{0}^{t} \left( e^{((1+\beta_{i})k_{i}-k_{i})W_{r}}\vee e^{((1+\gamma_{j})k_{j}-k_{i})W_{r}} \right) \nonumber\\
			& \times \left( \left\|e^{\lambda r} S_{r}f_{i} \right\|^{\beta_{i}}_{\infty}+\left\|e^{\lambda r} S_{r}f_{i} \right\|^{\gamma_{j}}_{\infty} \right) dr \Bigg]^{\frac{-1}{\beta_{i}+\gamma_{j}+1}},
			\end{align*}
			where $\left\lbrace j \right\rbrace=\left\lbrace1,2\right\rbrace/\{i\} \ \ i=1,2$. Proceeding in the same way as in the proof of Theorem $\ref{t2}$, we get  
			\begin{eqnarray}
				v_{1}(t,x) = \mathcal{T}_{1}[v_{2},v_{1}](t,x), \quad v_{2}(t,x) = \mathcal{T}_{2}[v_{1},v_{2}](t,x) ,\nonumber
			\end{eqnarray}
			whenever $t\leq \tau_{**}$ and $x\in D$. Moreover,
			\begin{align*}
			v_{i}(t,x) \leq& \exp\{{\lambda t}\} S_{t}f_{i}(x) \Bigg[ 1-\left(\beta_{i}+\gamma_{j}+1 \right) \int_{0}^{t} \left( e^{((1+\beta_{i})k_{i}-k_{i})W_{r}}\vee e^{((1+\gamma_{j})k_{j}-k_{i})W_{r}} \right)\\
			&\qquad\qquad\qquad\qquad \qquad\qquad \times \left( \left\| \exp\{{\lambda r}\} S_{r}f_{i} \right\|_{\infty}^{\beta_{i}}+\left\| \exp\{{\lambda r}\} S_{r}f_{i} \right\|_{\infty}^{\gamma_{j}} \right) dr  \Bigg]^{\frac{-1}{\beta_{i}+\gamma_{j}+1}} \nonumber\\
			=& \frac{L_{i}\psi(x)}{\left[ 1-\left( \beta_{i}+\gamma_{j}+1 \right) \left(  L^{\beta_{i}}_{i} \left\|\psi \right\|^{\beta_{i}}_{\infty}+  L^{\gamma_{j}}_{i}\left\| \psi \right\|^{\gamma_{j}}_{\infty} \right)\displaystyle\int_{0}^{t} \left( e^{((1+\beta_{i})k_{i}-k_{i})W_{r}}\vee e^{((1+\gamma_{j})k_{j}-k_{i})W_{r}} \right) dr \right]^{\frac{1}{\beta_{i}+\gamma_{j}+1}}}, \nonumber
			\end{align*}
			by the choices of $f_{1}$ and $f_{2}$.
		\end{proof}
	\end{theorem}
	\begin{corollary}\label{cor1}
		Let the random time $\tau'$ be defined by 
		\begin{align}
			\tau'=\inf \Bigg\{ t\geq 0 : \int_{0}^{t} &\max\Big\{\exp\{{(1+\gamma_{2})k_{2}-k_{1}W_{r}}\}, \exp\{{(1+\beta_{1})k_{1}-k_{1}W_{r}}\},\nonumber\\&\qquad\exp\{{(1+\gamma_{1})k_{1}-k_{2}W_{r}}\},\exp\{{(1+\beta_{2})k_{2}-k_{2}W_{r}}\} \Big\} dr \nonumber\\
			&\geq \min \Bigg\{ \frac{1}{(\beta_{1}+\gamma_{2}+1)(L^{\beta_{1}}_{1}\left\|\psi \right\|_{\infty}^{\beta_{1}}+L^{\gamma_{2}}_{1}\left\|\psi \right\|_{\infty}^{\gamma_{2}})},\nonumber \\ &\quad\frac{1}{(\gamma_{1}+\beta_{2}+1)(L^{\gamma_{1}}_{2}\left\|\psi \right\|_{\infty}^{\gamma_{1}}+L^{\beta_{2}}_{2}\left\|\psi \right\|_{\infty}^{\beta_{2}})} \Bigg\} \Bigg\}. \nonumber
		\end{align}
		Then $\tau'\leq \tau_{\ast \ast}$.
	\end{corollary}
	To estimate upper bounds for $\tau,$ with non-negative initial data $f_{i}, \ i=1,2$, we first notice that a sub-solution of equation (\ref{a5}) is given by the solution of  
	\begin{equation}\label{ab1} 
		\left\{
		\begin{aligned} 
			\frac{dh_{1}(t)}{dt}&=\exp\{({(1+\beta_{1})k_{1}-k_{1}) W_{t}}\} h_{1}^{1+\beta_{1}}(t)+\exp\{({(1+\gamma_{2})k_{2}-k_{1}) W_{t}}\} h_{2}^{1+\gamma_{2}}(t),\\ \frac{dh_{2}(t)}{dt}&= \exp\{({(1+\gamma_{1})k_{1}-k_{2}) W_{t}}\} h_{1}^{1+\gamma_{1}}(t)+\exp\{({(1+\beta_{2})k_{2}-k_{2}) W_{t}}\}h_{2}^{1+\beta_{2}}(t),\\
			h_{i}(0)&=v_{i}(0,\psi), \ i=1,2.  
		\end{aligned}
		\right.
	\end{equation}
	Implementing the same idea as in the proof of Theorem \ref{thm4.1} to the system (\ref{ab1}) with\\ $\beta_{1}=\gamma_{2}=\gamma_{1}=\beta_{2}=m\ (say)>0$, we deduce that 
	\begin{align}
		\frac{dE(t)}{dt}&\geq \mbox{min} \Big\{  \exp\{{((1+m)k_{1}-k_{1}) W_{t}}\}, \exp\{{((1+m)k_{2}-k_{1}) W_{t}}\},\nonumber\\  &\qquad\qquad \exp\{{((1+m)k_{1}-k_{2}) W_{t}}\}, \exp\{{((1+m)k_{2}-k_{2}) W_{t}}\} \Big\} 2^{-m} E^{1+m}(t). \nonumber
	\end{align} 
	If $\beta_{1}=\gamma_{1}=\beta, \ \gamma_{2}=\beta_{2}=\gamma$ with $\beta>\gamma>0,$ we obtain
	\begin{align}
	\frac{dE(t)}{dt}\geq 2\mbox{min} \Big\{  &\exp\{{((1+\beta)k_{1}-k_{1}) W_{t}}\}, \exp\{{((1+\gamma)k_{2}-k_{1}) W_{t}}\}, \exp\{{((1+\beta)k_{1}-k_{2}) W_{t}}\} \nonumber\\
	&\exp\{{((1+\gamma)k_{2}-k_{2}) W_{t}}\} \Big\} \left[2^{-(1+\gamma)}\epsilon_{0}E^{1+\gamma}(t)-D_{1}\epsilon_{0}^{\frac{1+\beta}{\beta-\gamma}} \right].\nonumber
	\end{align}
	For $\beta_{1}>\gamma_{2}>\gamma_{1}>\beta_{2}>0,$ we find
	\begin{align}
	\frac{dE(t)}{dt}\geq \mbox{min} \Big\{  &\exp\{{((1+\beta_{1})k_{1}-k_{1}) W_{t}}\}, \exp\{{((1+\gamma_{2})k_{2}-k_{1}) W_{t}}\}, \exp\{{((1+\gamma_{1})k_{1}-k_{2}) W_{t}}\} \nonumber\\
	&\exp\{{((1+\beta_{2})k_{2}-k_{2}) W_{t}}\} \Big\} \left[ \frac{\epsilon_{0}E^{1+\beta_{2}}(t)}{2^{(1+\gamma_{2})}} -\left(\epsilon_{0}^{\frac{1+\beta_{1}}{\beta_{1}-\gamma_{2}}} D_{2} + \epsilon_{0}^{\frac{1+\gamma_{1}}{\gamma_{1}-\beta_{2}}}D_{3}+\epsilon_{0}^{\frac{ 1+\gamma_{2}}{\gamma_{2}-\beta_{2}}}D_{4}\right)\right], \nonumber
	\end{align} 
	which also follows from the proof of Theorem \ref{thm4.1}. In this manner we obtain the upper bounds for blow-up time to the solutions of the system \eqref{b1}-\eqref{b2} as follows:
	\begin{theorem}\label{thm5.2}
		Let $\beta_{1},\beta_{2},\gamma_{1},\gamma_{2}>0$ and for any initial data $f_{i} \geq 0, \ i=1,2$.
		\begin{itemize}
		\item[1.] If $\beta_{1}=\gamma_{2}=\gamma_{1}=\beta_{2}=m\ (say)>0$ then $\tau \leq \tau^{\ast \ast}$, where 
		\begin{align}
		\tau^{\ast \ast}= \inf \Bigg\{ &t\geq 0 : \int_{0}^{t} \min \{\exp\{{((1+m)k_{1}-k_{1}) W_{s}}\}, \exp\{{((1+m)k_{2}-k_{1}) W_{s}}\},\nonumber\\ 
		&\qquad \exp\{{((1+m)k_{1}-k_{2}) W_{s}}\}, \exp\{{((1+m)k_{2}-k_{2}) W_{s}}\} \} ds\geq 2^{m} {m}^{-1}E^{-m}(0) \Bigg\}. \nonumber 
		\end{align}   
		\item[2.] 	If $\beta_{1}=\gamma_{1}=\beta, \ \gamma_{2}=\beta_{2}=\gamma,$ with $\beta>\gamma>0,$ and let \begin{align*}
		\epsilon_{0}&\leq \min \Bigg\{ 1, \left( h_{1}(0)/D_{1}^{1/1+\gamma} \right)^{\beta-\gamma} \Bigg\}.
		\end{align*} 
		Assume that 
		\begin{align*}
		2^{-(1+\gamma)}\epsilon_{0}E^{1+\gamma}(0)\geq \epsilon_{0}^{\frac{1+\beta}{\beta-\gamma}}D_{1}.
		\end{align*} then $\tau \leq \tau^{\ast \ast},$ where 
		\begin{align}
	    \tau^{\ast \ast} = \inf \Bigg\{ t\geq 0 : &\int_{0}^{t} \min \{\exp\{{((1+\beta)k_{1}-k_{1}) W_{s}}\}, \exp\{{((1+\gamma)k_{2}-k_{1}) W_{s}}\},\nonumber\\ &\qquad \exp\{{((1+\beta)k_{1}-k_{2}) W_{s}}\}, \exp\{{((1+\gamma)k_{2}-k_{2}) W_{s}}\} \} ds \nonumber\\&\quad\left.\geq {{\gamma}^{-1}E^{-\gamma}(0)\left(\frac{\epsilon_0}{2^{1+\gamma}}-\frac{\epsilon^{\frac{1+\beta}{\beta-\gamma}}_{0}D_{1}}{E^{1+\gamma}(0)} \right)^{-1}} \right\}.  \nonumber
		\end{align}   
		\item[3.] If $\beta_{1}>\gamma_{2}>\gamma_{1}>\beta_{2}>0,$ and let 
		\begin{align*}
		\epsilon_{0}&\leq\min \Bigg\{ 1, \left( h_{1}(0)/ D_{2}^{1/1+\gamma_{2}} \right)^{\beta_{1}-\gamma_{2}}, \left( h_{1}(0)/D_{3}^{1/1+\beta_{2}}\right)^{\gamma_{1}-\beta_{2}}\Bigg\}.
		\end{align*} 
		Assume that 
		\begin{eqnarray*} 
		2^{-(1+\gamma_{2})}\epsilon_{0}E^{1+\beta_{2}}(0)\geq \epsilon_{0}^{\frac{1+\beta_{1}}{\beta_{1}-\gamma_{2}}} D_{2} + \epsilon_{0}^{\frac{1+\gamma_{1}}{\gamma_{1}-\beta_{2}}}D_{3}+\epsilon_{0}^{\frac{ 1+\gamma_{2}}{\gamma_{2}-\beta_{2}}}D_{4}.
		\end{eqnarray*} then $\tau \leq \tau^{\ast \ast},$ where 
		\begin{align}
		\tau^{\ast \ast} = \inf \Bigg\{ t\geq 0 : &\int_{0}^{t} \min \{\exp\{{((1+\beta_{1})k_{1}-k_{1}) W_{s}}\}, \exp\{{((1+\gamma_{2})k_{2}-k_{1}) W_{s}}\},\nonumber\\ &\qquad \exp\{{((1+\gamma_{1})k_{1}-k_{2}) W_{s}}\}, \exp\{{((1+\beta_{2})k_{2}-k_{2}) W_{s}}\} \} ds \nonumber\\&\quad\left.\geq {{\beta_{2}}^{-1}E^{-\beta_{2}}(0)\left[ \frac{\epsilon_{0}}{2^{(1+\gamma_{2})}} -\frac{\left(  \epsilon_{0}^{\frac{1+\beta_{1}}{\beta_{1}-\gamma_{2}}} D_{2} + \epsilon_{0}^{\frac{1+\gamma_{1}}{\gamma_{1}-\beta_{2}}}D_{3}+\epsilon_{0}^{\frac{ 1+\gamma_{2}}{\gamma_{2}-\beta_{2}}}D_{4}\right)}{E^{1+\beta_{2}}(0)}\right]^{-1}} \right\}.  \nonumber
		\end{align}
		where $D_{1},D_{2},D_{3}$ and $D_{4}$ are defined in Theorem \ref{thm4.1}. 
			\end{itemize} 
	\end{theorem}
	\begin{corollary}
		Assume that $f_{i} \geq 0, \ i=1,2$ such that for $t \geq 0$ 
		\begin{align} {\label{cor3.1}}  
		\left(\beta_{i}+\gamma_{j}+1 \right) \int_{0}^{t} \left( e^{((1+\beta_{i})k_{i}-k_{i})W_{r}}\vee e^{((1+\gamma_{j})k_{j}-k_{i})W_{r}} \right) \left( \left\| e^{\lambda r} S_{r}f_{i} \right\|_{\infty}^{\beta_{i}}+\left\| e^{\lambda r} S_{r}f_{i} \right\|_{\infty}^{\gamma_{j}} \right) dr <1. 
		\end{align}
		Then the system \eqref{b1}-\eqref{b2} admits a gobal solutions $u_{i}(t,x)$ that satisfies 
		\begin{align*}
		0\leq u_{i}(t,x)\leq e^{k_{i}W_{t}+\lambda t} S_{t}f_{i}(x) \Bigg[& 1-\left(\beta_{i}+\gamma_{j}+1 \right) \int_{0}^{t} \left( e^{((1+\beta_{i})k_{i}-k_{i})W_{r}}\vee e^{((1+\gamma_{j})k_{j}-k_{i})W_{r}} \right) \nonumber\\ &\times \left( \left\| \exp\{{\lambda r}\} S_{r}f_{i} \right\|_{\infty}^{\beta_{i}}+\left\| \exp\{{\lambda r}\} S_{r}f_{i} \right\|_{\infty}^{\gamma_{j}} \right) dr  \Bigg]^{\frac{-1}{\beta_{i}+\gamma_{j}+1}}, \ t \geq 0,	
		\end{align*}
		where $\left\lbrace j \right\rbrace=\left\lbrace1,2\right\rbrace/\{i\}$.
	\end{corollary}
	\begin{proof}
		Proof of this result immediately follows from Theorem $\ref{thm5.1}.$  
	\end{proof}
	\section{Probability of blow-up}{\label{sec6}}
	In this section, we consider the system $(\ref{c3})$ with the set of parameters as follows:  
	\begin{equation}\label{s1}
		\begin{aligned}
			\varphi_{i}=\psi,  \ \ i=1,2, \ \ \ &k:=\frac{ k_{1}^{2}+k_{2}^{2}}{2}, \ \ \mu := \lambda+k,\ V_1=V_2=0,\\
			\mbox{ and } \ E(0)=\int_{D} [&f_{1}(x)+f_{2}(x)] \psi(x)dx.
		\end{aligned}
	\end{equation}	
for any positive initial data $f_{i},\ i=1,2.$ Here, we estimate  lower and upper bounds for the probability of the blow-up solution of the system \eqref{b1}-\eqref{b2} by using the methods developed in \cite{doz2010, Eug2017} and \cite{niu2012}. 

	Using the fact that $\psi(x)=0$ for $x \in \partial D,$ we have
	\begin{equation}{\label{as6}}
	\begin{aligned}
	v_{i}(s,\Delta \psi)=\int_{D} v_{i}(s,x)\Delta\psi(x)dx&=\int_{D} \Delta v_{i}(s,x)\psi(x)dx=\Delta v_{i}(s,\psi)=-\lambda v_{i}(s,\psi) .
	\end{aligned}
	\end{equation}
	We recall that $X(\alpha,\delta)$ is said to be \emph{a Gamma random variable} with parameters $\alpha+1>0,\delta>0$ if its density is given by (cf. \cite{li})
	\begin{equation} \label{abc2}
	\widetilde{f}(x) = \left\{
	\begin{aligned}
	&\frac{x^{\alpha}}{\delta^{\alpha+1}\Gamma (\alpha+1)} \exp\left\lbrace-\frac{x}{\delta} \right\rbrace , \ x\geq 0 , \\
	&\hspace{.5 in} 0, \hspace{1.05 in} \ x<0. 
	\end{aligned}
	\right.
	\end{equation}
	\begin{lemma}[\cite{yor2005},\cite{revuz1999}, Chapter 6, Corollary 1.2, \cite{yor2001}]{\label{Rll1}}
	For any $\alpha>0,$ the exponential functional $\displaystyle\int_{0}^{\infty} e^{2(B(t)-\alpha t)} dt$ 	is distributed as $(2X(\alpha,1))^{-1},$ where $B(t)$ is an one-dimensional standard Brownian motion. 
	\end{lemma}
		
The following result provides a lower bound: 
	\begin{theorem} {\label{thm6.1}}
	For each positive initial values $f_{i}\ (i=1,2)$ for $E(0)$ and $\mu$ are defined in $\eqref{s1},$ for $\rho_{1}$ and $\rho_{2}$ defined in $(\ref{a4})$ and $ \epsilon_{0}, D_{1},D_{2},D_{3}$ and $D_{4}$ are defined in Theorem \ref{thm4.1}, we have the following results:
	\begin{itemize}
		\item[1.] If $\beta_{1}=\gamma_{2}=\gamma_{1}=\beta_{2}=m\ (say),$ with $m>0$, implies that $\rho_{1}=\rho_{2}=\rho\ (say)$, then a lower bound for the  probability of blow-up solution of the system \eqref{b1}-\eqref{b2} is given by
		\begin{align}
		\mathbb{P}\{ \tau<\infty\}\geq \int^{\infty}_{\frac{2^{m} E^{-m}(0)}{m}}h_{1}(y)dy, \nonumber 
		\end{align}
		where $\tau \leq \tau^{\ast}_{1},$ 
		\begin{align*}
		\tau^{\ast}_{1}&=\inf \left\lbrace t\geq 0 : \int_{0}^{t} \exp\{{\rho W_{s}}-\mu m s\} ds \geq 2^{m}{m}^{-1}E^{-m}(0) \right\rbrace\ \mbox{and}\ \\
		h_{1}(y)&=\frac{(2/\rho^{2}y)^{(2\mu m /\rho^{2})}}{y\Gamma(2\mu m /\rho^{2})} \exp\left(-\frac{2}{\rho^{2}y}\right).\end{align*}

		\item[2.] If $\beta_{1}=\gamma_{1}=\beta, \ \gamma_{2}=\beta_{2}=\gamma$ with $\beta>\gamma>0,\ k_2\geq k_1$, 
        then a lower bound for the probability of blow-up solution of the system \eqref{b1}-\eqref{b2} is given by 
		\begin{align}
	    \mathbb{P}\{ \tau<\infty\}\geq  \frac{8 \mu \gamma}{\rho^{2}_{2}} \sum_{n \geq 1} \frac{\exp \Bigg\{  - \frac{\rho^{2}_{2} \widetilde{N}_{1}}{8}j^{2}_{\frac{2 \mu \gamma}{\rho^{2}_{2}}-1, n} \Bigg\}}{j^{2}_{\frac{2 \mu \gamma}{\rho^{2}_{2}}-1, n}}, \nonumber 
		\end{align}
		where $\tau \leq \tau_{2}^{\ast \ast},$
		\begin{align*}
		\tau^{\ast \ast}_{2} &= \inf \left\{ t\geq 0 : \int_{0}^{t} e^{-(\rho_{2} W_{s}-\mu \gamma s)}\mathbf{1} _{\{\rho_{2}W_{s}-\mu \gamma s \geq 0\}} ds \geq \widetilde{N}_{1} \right\},  \nonumber\\
        \widetilde{N}_{1}&={{\gamma}^{-1}E^{-\gamma}(0)\left(\frac{\epsilon_0}{2^{1+\gamma}}-\frac{\epsilon^{\frac{1+\beta}{\beta-\gamma}}_{0}D_{1}}{E^{1+\gamma}(0)} \right)^{-1}},
		\end{align*}
        and $\Big\{j_{\frac{2 \mu \gamma}{\rho^{2}_{2}}-1, n}\Big\}_{n \geq 1}$ is the increasing sequence of all positive zeros of the Bessel function of the first kind of order $\frac{2 \mu}{\rho^{2}_{2}}-1>-1$ (see more details in \cite{Eug2017}).
	
		\item[3.] If $\beta_{1}>\gamma_{2}>\gamma_{1}>\beta_{2}>0$ with $k_2\geq k_1$,
         then a lower bound for the probability of blow-up solution of the system \eqref{b1}-\eqref{b2} is given by 
		\begin{align}
	    \mathbb{P}\{ \tau<\infty\}\geq  \frac{8 \mu \beta_{2}}{\rho^{2}_{2}} \sum_{n \geq 1} \frac{\exp \Bigg\{  - \frac{\rho^{2}_{2}\widetilde{N}_{2}}{8}j^{2}_{\frac{2 \mu \beta_{2}}{\rho^{2}_{2}}-1, n} \Bigg\}}{j^{2}_{\frac{2 \mu \beta_{2}}{\rho^{2}_{2}}-1, n}}, \nonumber 
		\end{align}
		where $\tau \leq \tau^{\ast \ast}_{3}$,
		\begin{align*}
		&\tau^{\ast \ast}_{3} = \inf \Bigg\{ t\geq 0 : \int_{0}^{t}  e^{\rho_{2} W_{s}-\mu \beta_{2}s}\mathbf{1} _{\{\rho_{2}W_{s}-\mu \gamma s \geq 0\}} ds \geq \widetilde{N}_{2} \Bigg\},  \nonumber  \\
		&\widetilde{N}_2={{\beta_{2}}^{-1}E^{-\beta_{2}}(0)\left[ \frac{\epsilon_{0}}{2^{(1+\gamma_{2})}} -\frac{\left(\epsilon_{0}^{\frac{1+\beta_{1}}{\beta_{1}-\gamma_{2}}}D_{2} + \epsilon_{0}^{\frac{1+\gamma_{1}}{\gamma_{1}-\beta_{2}}}D_{3}+\epsilon_{0}^{\frac{ 1+\gamma_{2}}{\gamma_{2}-\beta_{2}}}D_{4}\right)}{E^{1+\beta_{2}}(0)}\right]^{-1}}.
		\end{align*}
	    \end{itemize}
		\end{theorem}
		
			
		\begin{proof}
		We deduce from $(\ref{c3})$ that
		\begin{align}{\label{se2}} 
		v_{i}(t,\psi) &= v_{i}(0,\psi) + \int_{0}^{t} \left[ v_{i}(s,\Delta\psi) - \frac{k_{i}^{2}}{2} v_{i}(s,\psi) \right] ds \nonumber\\
		&\qquad + \int_{0}^{t} \exp\{{-k_{i} W_{s}}\} (\exp\{{k_{i} W_{s}}\} v_{i})^{1+\beta_{i}} (s,\psi) ds\nonumber\\
		&\qquad + \int_{0}^{t} \exp\{{-k_{i} W_{s}}\} (\exp\{{k_{j} W_{s}}\} v_{j})^{1+\gamma_{j}} (s,\psi) ds,   
		\end{align}
		where $\left\lbrace j \right\rbrace=\left\lbrace1,2\right\rbrace/\{i\}$.
		Therefore, $(\ref{se2})$ becomes
		\begin{align}{\label{s2}} 
		\frac{dv_{i}(t,\psi)}{dt} &= - \left( \lambda + \frac{k_{i}^{2}}{2}\right)  v_{i}(t,\psi) + \exp\{{-k_{i} W_{t}}\}\left(\exp\{k_{i}W_{t}\}v_{i}\right)^{1+\beta_{i}} (t,\psi) \nonumber\\
		&\qquad+ \exp\{{-k_{i} W_{t}}\}\left(\exp\{k_{j}W_{t}\}v_{j}\right)^{1+\gamma_{j}} (t,\psi),
		\end{align}	
	for a.e. $t\geq 0$. 	By utilizing Jensen's inequality and $(\ref{s1})$ in $(\ref{s2}),$ we have  
		\begin{align} 
		\frac{dv_{i}(t,\psi)}{dt} &\geq - \mu v_{i}(t,\psi)+ \exp\{{-k_{i} W_{t}+(1+\beta_{i})k_{i}W_{t}}\} {v_{i} (t,\psi)}^{1+\beta_{i}} \nonumber\\
		&\qquad+ \exp\{{-k_{i} W_{t}+(1+\gamma_{j})k_{j}W_{t}}\} {v_{j} (t,\psi)}^{1+\gamma_{j}}.\nonumber  
		\end{align}
	Therefore, $v_{i}(t,\psi) \geq h_{i}(t),$ for $ i=1,2,$ where
		\begin{equation} 
		\left\{
		\begin{aligned}
		\frac{dh_{1}(t)}{dt}&=-\mu h_{1}(t) +\exp\{{\rho_{1} W_{t}}\} \left[ h_{1}^{1+\beta_{1}}(t)+h_{2}^{1+\gamma_{2}}(t) \right], \\ 
		\frac{dh_{2}(t)}{dt}&=-\mu h_{2}(t) +\exp\{{\rho_{2} W_{t}}\} \left[ h_{1}^{1+\gamma_{1}}(t)+h_{2}^{1+\beta_{2}}(t) \right],\\
		h_{i}(0)&=v_{i}(0,\psi).  
		\end{aligned}
		\right.
		\end{equation}
		Let  us define $E(t):=h_{1}(t)+h_{2}(t), \ t\geq 0,$ so that $E(\cdot)$ satisfies 
		\begin{equation} 
		\frac{dE(t)}{dt}=-\mu E(t)+\exp\{{\rho_{1} W_{t}}\} \left[ h_{1}^{1+\beta_{1}}(t)+h_{2}^{1+\gamma_{2}}(t) \right] + \exp\{{\rho_{2} W_{t}}\} \left[ h_{1}^{1+\gamma_{1}}(t)+h_{2}^{1+\beta_{2}}(t) \right].
		\end{equation}
		\textbf{Proof of (1).} If $\beta_{1}=\gamma_{2}=\gamma_{1}=\beta_{2}=m$ (say) and $m>0,$ implies that $\rho_{1}=\rho_{2}=\rho$ (say), then by using the same idea used as in Case 1 of Theorem $\ref{thm4.1},$ we have 
		\begin{align}
		\frac{dE(t)}{dt}\geq-\mu E(t)+2^{-m}\exp\{{\rho W_{t}}\} E^{1+m}(t). \nonumber
		\end{align}
		Thus, $E(t)$ blows up not later than the solution $I(t)$ of the equation
		\begin{align}
		\frac{dI(t)}{dt}=-\mu I(t)+2^{-m}\exp\{{\rho W_{t}}\} I^{1+m}(t),\ I(0)=E(0), \nonumber
		\end{align}
		and 
		\begin{align}
		I(t)= e^{-\mu t}\left\lbrace E^{-m}(0)-m2^{-m} \displaystyle\int_{0}^{t} \exp\{\rho W_{s}-\mu m s\}   ds \right\rbrace^{-\frac{1}{m}}.\nonumber 
		\end{align}
		For the above inequality, the blow-up time is given by
		\begin{eqnarray}\label{tau1}
		\tau^{\ast}_{1}=\inf \left\lbrace t\geq 0 : \int_{0}^{t} \exp\{{\rho W_{s}}-\mu m s\} ds \geq 2^{m}{m}^{-1}E^{-m}(0) \right\rbrace. 
		\end{eqnarray} 
	    Note that $\tau \leq \tau^{\ast}_{1}$ from the definition of $\tau^{\ast }_{1},$  and 
		\begin{align}
		\mathbb{P}\{ \tau^{\ast}_{1}=\infty\} &= \mathbb{P}\Bigg( \int_{0}^{\infty} \exp\{\rho W_{s}-\mu m s\}ds\leq 2^{m} {m}^{-1}E^{-m}(0) \Bigg)  \nonumber\\
		&= \mathbb{P}\Bigg( \int_{0}^{\infty} \exp\{\rho W_{s}^{(\hat{\alpha})}\}ds \leq 2^{m} {m}^{-1}E^{-m}(0) \Bigg)  ,\nonumber
		\end{align}
		where $W_{s}^{(\hat{\alpha})}:=W_{s}-\hat{\alpha}s$ and $\hat{\alpha} =\frac{\mu m}{\rho}.$ By performing the transformation $s\mapsto \frac{4t}{\rho^2}$ and setting $\nu = \frac{2\hat{\alpha}}{\rho}$,  we get
		\begin{align}
		\mathbb{P}\{ \tau<\infty\} \geq \mathbb{P}\{ \tau^{\ast }_{1}<\infty\}=1-\mathbb{P}\{ \tau^{\ast }_{1}=\infty\}
		&=\mathbb{P}\Bigg(\frac{4}{\rho^{2}} \int_{0}^{\infty} \exp\{2 W_{t}^{(\nu)}\}dt > 2^{m} {m}^{-1}E^{-m}(0) \Bigg).\nonumber
		\end{align}
		It follows from Chapter 6, Corollary 1.2, \cite{yor2001}, that 
		\begin{equation}{\label{ex1}}
		\int_{0}^{\infty} \exp\{ 2W_{t}^{(\nu)}\}dt \overset{Law}{=} \frac{1}{2Z_{\nu}},
		\end{equation}
		where $Z_{\nu}$ is a Gamma random variable with parameter $\nu,$ that is, $\mathbb{P} (Z_{\nu} \in dy)= \frac{1}{\Gamma(\nu)}e^{-y} y^{\nu-1} dy.$ Therefore, we get 
		\begin{align}
		\mathbb{P}\{ \tau<\infty\}\geq \int^{\infty}_{\frac{2^{m} E^{-m}(0)}{m}} h_{1}(y)dy, \nonumber 
		\end{align}
		where 
		\begin{align}
		h_{1}(y)=\frac{(2/\rho^{2}y)^{(2\mu m /\rho^{2})}}{y \Gamma(2\mu m/\rho^{2})} \exp\left(-\frac{2}{\rho^{2}y}\right).  \nonumber 
		\end{align}
		
		\noindent\textbf{{Proof of (2).}} Suppose $\beta_{1}=\gamma_{1}=\beta, \ \gamma_{2}=\beta_{2}=\gamma$ with $\beta>\gamma>0$. By using the same method as in Case 2 of Theorem $\ref{thm4.1},$ we get 
		\begin{align}
		\frac{dE(t)}{dt}\geq-\mu E(t)+\left( e^{\rho_{1} W_{s}} \wedge e^{\rho_{2} W_{s}}\right)\left[2^{-(1+\gamma)}\epsilon_{0}E^{(1+\gamma)}(t)-D_{1}\epsilon_{0}^{\frac{1+\beta}{\beta-\gamma}}\right]. \nonumber
		\end{align}
		Thus, $E(t)$ blows up not later than the solution $I(t)$ of the equation
		\begin{align}
		\frac{dI(t)}{I^{(1+\gamma)}(t)}=\frac{-\mu}{I^{\gamma}(t)}dt+\left( e^{\rho_{1} W_{s}} \wedge e^{\rho_{2} W_{s}}\right)\left[2^{-(1+\gamma)}\epsilon_{0}-\frac{D_{1}\epsilon_{0}^{\frac{1+\beta}{\beta-\gamma}}}{I^{(1+\gamma)}(0)}\right]dt,\ I(0)=E(0)\nonumber
		\end{align}
		and 
		\begin{align}
		I(t)= e^{-\mu t}\left\lbrace E^{-\gamma}(0)-\gamma\left[ \frac{\epsilon_0}{2^{1+\gamma}}-\frac{ \epsilon_{0}^{\frac{1+\beta}{\beta-\gamma}}D_{1} }{E^{1+\gamma}(0)} \right] \displaystyle\int_{0}^{t} \left( e^{\rho_{1} W_{s}} \wedge e^{\rho_{2} W_{s}} \right) e^{-\mu \gamma s} ds \right\rbrace^{-\frac{1}{\gamma}}.\nonumber 
		\end{align}
		For the above inequality, the blow-up time is given by		
		\begin{align}
		\tau^{\ast}_{2} = \inf \left\{ t\geq 0 : \int_{0}^{t} \left( e^{\rho_{1} W_{s}} \wedge e^{\rho_{2} W_{s}}\right) e^{-\mu \gamma s} ds \geq {{\gamma}^{-1}E^{-\gamma}(0)\left(\frac{\epsilon_0}{2^{1+\gamma}}-\frac{\epsilon^{\frac{1+\beta}{\beta-\gamma}}_{0}D_{1}}{E^{1+\gamma}(0)} \right)^{-1}} \right\},  \nonumber
		\end{align}
		where $\rho_{1}$ and $\rho_{2}$ are defined in (\ref{a4}).
		Note that $k_2\geq k_1$ implies that $\rho_{2}>0,$ $\rho_2\leq\rho_1$ and 
		\begin{align}
		\int_{0}^{t} \left( e^{\rho_{1} W_{s}} \wedge e^{\rho_{2} W_{s}}\right) e^{-\mu \gamma s} ds &= \int_{0}^{t} e^{\rho_{1} W_{s}-\mu \gamma s}  \mathbf{1} _{\{W_{s}<0\}} ds+\int_{0}^{t} e^{\rho_{2} W_{s}-\mu \gamma s}  \mathbf{1} _{\{W_{s} \geq 0\}} ds \nonumber\\
		&\geq \int_{0}^{t} e^{\rho_{2} W_{s}-\mu \gamma s}\mathbf{1} _{\{W_{s} \geq 0\}} ds, \nonumber
		\end{align}
		for all $t \geq 0$. We conclude that 
		\begin{align}\label{f1}
		\int_{0}^{t} \left( e^{\rho_{1} W_{s}} \wedge e^{\rho_{2} W_{s}}\right) e^{-\mu \gamma s} ds \geq \int_{0}^{t} e^{-(\rho_{2} W_{s}-\mu \gamma s)}\mathbf{1} _{\{\rho_{2}W_{s}-\mu \gamma s \geq 0\}} ds.
		\end{align}
		Therefore, it follows that 
		\begin{align}
		\tau^{\ast \ast}_{2} = \inf \left\{ t\geq 0 : \int_{0}^{t} e^{-(\rho_{2} W_{s}-\mu \gamma s)}\mathbf{1} _{\{\rho_{2}W_{s}-\mu \gamma s \geq 0\}} ds \geq {{\gamma}^{-1}E^{-\gamma}(0)\left(\frac{\epsilon_0}{2^{1+\gamma}}-\frac{\epsilon^{\frac{1+\beta}{\beta-\gamma}}_{0}D_{1}}{E^{1+\gamma}(0)} \right)^{-1}} \right\}.  \nonumber
		\end{align}
	It can be easily seen that $\tau_2^{*} \leq \tau_2^{**} $. From Theorem \ref{thm4.1} (2), it is clear that $\tau \leq \tau_2^{*} \leq  \tau^{\ast \ast}_{2}$, and 	using the definition of $\tau^{\ast \ast}_{2},$ and Theorem 2.6 in \cite{Eug2017}, we have
		\begin{align}
	\mathbb{P}\{ \tau<\infty\} \geq \mathbb{P}\{ \tau^{\ast\ast}_{2}<\infty\} 
		&= \mathbb{P}\Bigg( \int_{0}^{\infty} e^{-(\rho_{2} W_{s}-\mu \gamma s)}\mathbf{1} _{\{\rho_{2}W_{s}-\mu \gamma s \geq 0\}} ds \geq \widetilde{N}_{1} \Bigg)  \nonumber\\
		&= \int^{\infty}_{ \widetilde{N}_{1}} \mu \gamma \sum_{n \geq 1} \exp \Bigg\{ -\left( \frac{\rho^{2}_{2}}{8}j^{2}_{\frac{2 \mu \gamma}{\rho^{2}_{2}}-1, n}\right)y \Bigg\}dy \nonumber\\
		&= \frac{8 \mu \gamma}{\rho^{2}_{2}} \sum_{n \geq 1} \frac{\exp \Bigg\{  - \frac{\rho^{2}_{2} \widetilde{N}_{1}}{8}j^{2}_{\frac{2 \mu \gamma}{\rho^{2}_{2}}-1, n} \Bigg\}}{j^{2}_{\frac{2 \mu \gamma}{\rho^{2}_{2}}-1, n}},\nonumber 
		\end{align}
	where $\Big\{j_{\frac{2 \mu \gamma}{\rho^{2}_{2}}-1, n}\Big\}_{n \geq 1}$ is the increasing sequence of all positive zeros of the Bessel function of the first kind of order $\frac{2 \mu}{\rho^{2}_{2}}-1>-1$ (for more details, see \cite{Eug2017}). Note that we have used the Monotone Convergence Theorem to obtain the final equality.
				
		\noindent \textbf{{Proof of (3).}} Suppose $\beta_{1}>\gamma_{2}>\gamma_{1}>\beta_{2}>0$. By using the same procedure as used in Theorem \ref{thm4.1} (Case 3), we obtain  
		\begin{align}
			\frac{dE(t)}{E^{(1+\beta_{2})}(t)}\geq\frac{-\mu}{E^{\beta_{2}}(t)}dt +\left( e^{\rho_{1} W_{s}} \wedge e^{\rho_{2} W_{s}}\right)\left[ \frac{\epsilon_{0}}{2^{(1+\gamma_{2})}} -\frac{\left(  \epsilon_{0}^{\frac{1+\beta_{1}}{\beta_{1}-\gamma_{2}}} D_{2} + \epsilon_{0}^{\frac{1+\gamma_{1}}{\gamma_{1}-\beta_{2}}}D_{3}+\epsilon_{0}^{\frac{ 1+\gamma_{2}}{\gamma_{2}-\beta_{2}}}D_{4}\right)}{E^{1+\beta_{2}}(0)}\right]dt. \nonumber
			\end{align}
			Thus, $E(t)$ blows up not later than the solution $I(t)$ of the equation
			\begin{align}
			\frac{dI(t)}{I^{(1+\beta_{2})}(t)}&=\frac{-\mu}{I^{\beta_{2}}(t)}dt+\left( e^{\rho_{1} W_{s}} \wedge e^{\rho_{2} W_{s}}\right)\left[ \frac{\epsilon_{0}}{2^{(1+\gamma_{2})}} -\frac{\left(  \epsilon_{0}^{\frac{1+\beta_{1}}{\beta_{1}-\gamma_{2}}} D_{2} + \epsilon_{0}^{\frac{1+\gamma_{1}}{\gamma_{1}-\beta_{2}}}D_{3}+\epsilon_{0}^{\frac{ 1+\gamma_{2}}{\gamma_{2}-\beta_{2}}}D_{4}\right)}{I^{1+\beta_{2}}(0)}\right]dt,\nonumber\\
			I(0)&=E(0), \nonumber
			\end{align}
			and 
			\begin{align}
			I(t)= e^{-\mu t}\Bigg\{ E^{-\beta_{2}}(0)-\beta_{2}&\left[ \frac{\epsilon_{0}}{2^{(1+\gamma_{2})}} -\frac{\left(  \epsilon_{0}^{\frac{1+\beta_{1}}{\beta_{1}-\gamma_{2}}} D_{2} + \epsilon_{0}^{\frac{1+\gamma_{1}}{\gamma_{1}-\beta_{2}}}D_{3}+\epsilon_{0}^{\frac{ 1+\gamma_{2}}{\gamma_{2}-\beta_{2}}}D_{4}\right)}{E^{1+\beta_{2}}(0)}\right]\nonumber\\
			&\qquad \times\int_{0}^{t} \left( e^{\rho_{1} W_{s}} \wedge e^{\rho_{2} W_{s}} \right) e^{-\mu \beta_{2} s} ds \Bigg\}^{-\frac{1}{\beta_{2}}}.\nonumber 
			\end{align}
			For the above inequality, the blow-up time is given by
		\begin{align}
		\tau^{\ast}_{3} = \inf \Bigg\{ t&\geq 0 : \int_{0}^{t}\left( e^{\rho_{1} W_{s}} \wedge e^{\rho_{2} W_{s}}\right) e^{-\mu \beta_{2}s} ds \geq \nonumber\\
		&\left. \quad{{\beta_{2}}^{-1}E^{-\beta_{2}}(0)\left[ \frac{\epsilon_{0}}{2^{(1+\gamma_{2})}} -\frac{\left(  \epsilon_{0}^{\frac{1+\beta_{1}}{\beta_{1}-\gamma_{2}}} D_{2} + \epsilon_{0}^{\frac{1+\gamma_{1}}{\gamma_{1}-\beta_{2}}}D_{3}+\epsilon_{0}^{\frac{ 1+\gamma_{2}}{\gamma_{2}-\beta_{2}}}D_{4}\right)}{E^{1+\beta_{2}}(0)}\right]^{-1}}\right\}.  \nonumber
		\end{align}
		Using the fact  \eqref{f1}, we have
		\begin{align}
		\tau^{\ast \ast}_{3} = \inf \Bigg\{ t\geq 0 : \int_{0}^{t}&  e^{-\left( \rho_{2} W_{s}-\mu \beta_{2}s\right) }\mathbf{1} _{\{\rho_{2}W_{s}-\mu \gamma s \geq 0\}} ds \geq {\beta_{2}}^{-1}E^{-\beta_{2}}(0)\nonumber\\
		&\left. \times \left[ \frac{\epsilon_{0}}{2^{(1+\gamma_{2})}} -\frac{\left(  \epsilon_{0}^{\frac{1+\beta_{1}}{\beta_{1}-\gamma_{2}}} D_{2} + \epsilon_{0}^{\frac{1+\gamma_{1}}{\gamma_{1}-\beta_{2}}}D_{3}+\epsilon_{0}^{\frac{ 1+\gamma_{2}}{\gamma_{2}-\beta_{2}}}D_{4}\right)}{E^{1+\beta_{2}}(0)}\right]^{-1}\right\}, \nonumber
	    \end{align}
	and     it can be easily seen that $\tau_3^{*} \leq \tau_3^{**} $. From Theorem \ref{thm4.1} (3), it is clear that $\tau \leq \tau_3^{*} \leq  \tau^{\ast \ast}_{3}$, and using the definition of $\tau^{\ast \ast}_{3},$ and Theorem 2.6 in \cite{Eug2017}, we have
		\begin{align}
	\mathbb{P}\{ \tau<\infty\}  \geq \mathbb{P}\{ \tau^{\ast\ast}_{3}<\infty\} &=\mathbb{P}\Bigg( \int_{0}^{\infty} e^{-\left(\rho_{2} W_{s}-\mu \beta_{2}s\right) }\mathbf{1} _{\{\rho_{2}W_{s}-\mu \beta_{2} s \geq 0\}}ds\geq \widetilde{N}_2 \Bigg)  \nonumber\\
		&= \int^{\infty}_{ \widetilde{N}_{2}} \mu \beta_{2} \sum_{n \geq 1} \exp \Bigg\{ -\left( \frac{\rho^{2}_{2}}{8}j^{2}_{\frac{2 \mu \beta_{2}}{\rho^{2}_{2}}-1, n}\right)y \Bigg\}dy \nonumber\\
		&= \frac{8 \mu \beta_{2}}{\rho^{2}_{2}} \sum_{n \geq 1} \frac{\exp \Bigg\{  - \frac{\rho^{2}_{2}\widetilde{N}_2}{8}j^{2}_{\frac{2 \mu \beta_{2}}{\rho^{2}_{2}}-1, n} \Bigg\}}{j^{2}_{\frac{2 \mu \beta_{2}}{\rho^{2}_{2}}-1, n}},\nonumber 
		\end{align}
which completes the proof. 
	\end{proof}
	The following theorem provides an upper bound for the probability of blow-up solution of the system \eqref{b1}-\eqref{b2}. Let us take 
	\begin{align}\label{new1}
	\left.
	\begin{aligned}
	A&=\min\Big\{(1+\beta_{1})k_{1}-k_{1}, (1+\beta_{2})k_{2}-k_{1}, (1+\gamma_{1})k_{2}-k_{1}, (1+\gamma_{2})k_{2}-k_{1}\Big\}, \\
	B&=\max\Big\{(1+\beta_{1})k_{1}-k_{1}, (1+\beta_{2})k_{2}-k_{1}, (1+\gamma_{1})k_{2}-k_{1}, (1+\gamma_{2})k_{2}-k_{1} \Big\},\\
    b_{1}&=\min\Big\{ \beta_{1}, \beta_{2},\gamma_{1} , \gamma_{2} \Big\}\ \mbox{and}\ \Lambda=\lambda+\frac{k^{2}_{1}\wedge k^{2}_{2}}{2} . 
	\end{aligned}
	\right\}   
	\end{align}

	\begin{theorem}
		Assume that $\beta_{1},\beta_{2},\gamma_{1},\gamma_{2}>0$ along with  $A>0$ and $V_{1}=V_{2}=0$. Let $f_{1}=L_{1}\psi$ and $f_{2}=L_{2}\psi,$ for some positive constants $L_{1}$ and $L_{2}$ with $L_{1}\leq L_{2}.$ Then an upper bound for the probability of blow-up solution of the system \eqref{b1}-\eqref{b2} is given by
		\begin{align}
		\mathbb{P}\{ \tau<\infty\}\leq \int_{\widetilde{N}_3}^{\infty} h_{3}(y)dy, \nonumber 
		\end{align}
		where $\tau'' \leq \tau$,
		\begin{align}
			\tau''&=\inf \Bigg\{ t\geq 0 : \int_{0}^{t} e^{B W_{r}-\Lambda b_{1} r} dr \geq \widetilde{N}_3\Bigg\}, \nonumber\\
		\widetilde{N}_3&= \min \Bigg\{ \frac{1}{(\beta_{1}+\gamma_{2}+1)(L^{\beta_{1}}_{1}\left\|\psi \right\|_{\infty}^{\beta_{1}}+L^{\gamma_{2}}_{1}\left\|\psi \right\|_{\infty}^{\gamma_{2}})}, \frac{1}{(\gamma_{1}+\beta_{2}+1)(L^{\gamma_{1}}_{2}\left\|\psi \right\|_{\infty}^{\gamma_{1}}+L^{\beta_{2}}_{2}\left\|\psi \right\|_{\infty}^{\beta_{2}})} \Bigg\}\nonumber\\&\qquad-\frac{1}{\Lambda b_{1}},\nonumber\\
		 h_{3}(y)&=\frac{(2/B^{2}y)^{(2\Lambda b_{1} /B^{2})}}{y \Gamma(2\Lambda b_{1}/B^{2})} \exp\left(-\frac{2}{B^{2}y}\right)\ and\ B,\ \Lambda,\ b_{1}\ \mbox{are defined in } \eqref{new1}.\nonumber
		\end{align}
		\end{theorem}
	\begin{proof}
		By using the same procedure as used in Theorem \ref{thm5.1}, we have 
		\begin{align}
		\tau'=\inf \Bigg\{ t\geq 0 : \int_{0}^{t} \left(e^{AW_{r}} \vee e^{BW_{r}} \right) e^{-\Lambda b_{1} r} dr &\geq \min \Bigg\{ \frac{1}{(\beta_{1}+\gamma_{2}+1)(L^{\beta_{1}}_{1}\left\|\psi \right\|_{\infty}^{\beta_{1}}+L^{\gamma_{2}}_{1}\left\|\psi \right\|_{\infty}^{\gamma_{2}})},\nonumber \\ &\quad\frac{1}{(\gamma_{1}+\beta_{2}+1)(L^{\gamma_{1}}_{2}\left\|\psi \right\|_{\infty}^{\gamma_{1}}+L^{\beta_{2}}_{2}\left\|\psi \right\|_{\infty}^{\beta_{2}})} \Bigg\} \Bigg\}. \nonumber
		\end{align}
		From the Corollary \ref{cor1} it is clear that $\tau' \leq \tau_{\ast \ast} \leq \tau$. Note that $0<A\leq B$ and so
		\begin{align}
		\int_{0}^{t} \left( e^{A W_{s}} \vee e^{B W_{s}}\right) e^{-\Lambda b_{1} s} ds &= \int_{0}^{t} e^{A W_{s}-\Lambda b_{1} s}  \mathbf{1} _{\{W_{s}<0\}} ds+\int_{0}^{t} e^{B W_{s}-\Lambda b_{1}s}  \mathbf{1} _{\{W_{s} \geq 0\}} ds \nonumber\\
		&\leq \int_{0}^{\infty} e^{-\Lambda b_{1} s} ds+\int_{0}^{t} e^{B W_{s}-\Lambda b_{1} s} ds \nonumber\\
		&=\frac{1}{\Lambda b_{1}}+\int_{0}^{t} e^{B W_{s}-\Lambda b_{1} s} ds,
		\end{align}
		for all $t \geq 0$.	Therefore, it follows that 	
		\begin{align}
		\tau''=\inf \Bigg\{ t\geq 0 : \int_{0}^{t} e^{B W_{r}-\Lambda b_{1} r} dr &\geq \min \Bigg\{ \frac{1}{(\beta_{1}+\gamma_{2}+1)(L^{\beta_{1}}_{1}\left\|\psi \right\|_{\infty}^{\beta_{1}}+L^{\gamma_{2}}_{1}\left\|\psi \right\|_{\infty}^{\gamma_{2}})},\nonumber \\ &\quad\frac{1}{(\gamma_{1}+\beta_{2}+1)(L^{\gamma_{1}}_{2}\left\|\psi \right\|_{\infty}^{\gamma_{1}}+L^{\beta_{2}}_{2}\left\|\psi \right\|_{\infty}^{\beta_{2}})} \Bigg\} -\frac{1}{\Lambda b_{1}}\Bigg\}, \nonumber
		\end{align}
	 and $\tau'' \leq \tau'$. Using the definition of $\tau''$, we have 
		\begin{align}
		\mathbb{P}\{ \tau<\infty\} \leq \mathbb{P}( \tau''<\infty) &= \mathbb{P}\Bigg(\int_{0}^{\infty} e^{B W_{r}-\Lambda b_{1} r} dr \geq \widetilde{N}_3\Bigg)= \mathbb{P}\Bigg( \int_{0}^{\infty} \exp\{B W_{s}^{(\hat{\alpha})}\}ds \geq \widetilde{N}_3 \Bigg), \nonumber
		\end{align}	
		where $W_{s}^{(\hat{\alpha})}:=W_{s}-\hat{\alpha}s$ and $\hat{\alpha} =\frac{\Lambda b_{1}}{B}.$ By performing the transformation $s\mapsto \frac{4t}{B^2}$ and setting $\nu = \frac{2\hat{\alpha}}{B}$,  we get
		\begin{align}
		\mathbb{P}\{ \tau<\infty\}
		&\leq \mathbb{P}\Bigg(\frac{4}{B^{2}} \int_{0}^{\infty} \exp\{2 W_{t}^{(\nu)}\}dt \geq \widetilde{N}_3 \Bigg).\nonumber
		\end{align}
		Therefore, we deduce 
		\begin{align}
		\mathbb{P}\{ \tau<\infty\}\leq \int_{\widetilde{N}_3}^{\infty} h_{3}(y)dy, \nonumber 
		\end{align}
		 which completes the proof. 
\end{proof}

	\section{Semilinear SPDEs with two-dimensional Brownian motions}{\label{sec7}}
	In this section, we consider the following  system of semilinear SPDEs driven by two-dimensional Brownian motions of the form:
	\begin{equation}\label{Rb1}
	\left\{
	\begin{aligned}
	du_{1}(t,x)=&\left[(\Delta+V_{1}) u_{1}(t,x)+C_{11} u_{1}^{1+\beta_{1}}(t,x)+C_{12}u^{1+\gamma_{2}}_{2}(t,x) \right]dt\\
	&\quad+k_{11}u_{1}(t,x)dW_{1}(t)+k_{12}u_{1}(t,x)dW_{2}(t), \\
	du_{2}(t,x)=&\left[(\Delta+V_{2}) u_{2}(t,x)+C_{21}u_{1}^{1+\gamma_{1}}(t,x)+C_{22}u_{2}^{1+\beta_{2}}(t,x) \right]dt\\
	&\quad+k_{21}u_{2}(t,x)dW_{1}(t)+k_{22}u_{2}(t,x)dW_{2}(t), 
	\end{aligned}
	\right.
	\end{equation}
	for $x \in D,t>0$, along	with the Dirichlet conditions
	\begin{align} \label{Rb2}
	u_{i}(0,x)&=f_{i}(x)\geq 0, \ \ x \in D ,\nonumber\\
	u_{i}(t,x)&=0, \ \ t\geq 0, \ \ x \in \partial D, \ \ i=1,2.
	\end{align}
	Here, $\beta_{1}\geq \gamma_{2}\geq \gamma_{1}\geq \beta_{2}>0$ are constants, $D\subset \mathbb{R}^{d}$ is a bounded and smooth domain, $C_{ij}, \ k_{ij}\neq 0$ are constants, $i,j=1,2,$ $f_{i}$ are of class $C^{2}$ and not identically zero for $i=1,2$ and $(W_{1}(\cdot),W_{2}(\cdot))$ is a standard two-dimensional Brownian motion with respect to filtered probability space $\left( \Omega, \mathscr{F}, (\mathscr{F}_{t})_{t \geq 0}, \mathbb{P} \right).$
	
	First, we consider the system \eqref{Rb1}-\eqref{Rb2} with the parameters $V_{i}=\lambda+\frac{k^{2}_{i1}+k_{i2}^{2}}{2},$
	\begin{equation}\label{Ra4}
		\left\{
	\begin{aligned}
	\eta_{i}:= \mbox{max}\{ C_{i1}, C_{i2}\},\ \ &\sigma_{i} =\min \{ C_{i1}, C_{i2} \} \ \ i=1,2, \\  (1+\beta_{1})k_{11}-k_{11}&=(1+\gamma_{2})k_{21}-k_{11} =:\rho_{11}, \\
	(1+\beta_{1})k_{12}-k_{12}&=(1+\gamma_{2})k_{22}-k_{12}=:\rho_{12}, \\ 
	(1+\beta_{2})k_{21}-k_{21}&=(1+\gamma_{1})k_{11}-k_{21}=:\rho_{21}, \\
	(1+\beta_{2})k_{22}-k_{22}&=(1+\gamma_{1})k_{12}-k_{22}=:\rho_{22}. 
	\end{aligned}
\right.
	\end{equation}
	Our aim is to find random times $\sigma_{*}$ and $\sigma^{*}$ such that $0\leq \sigma_*\leq \tau\leq \sigma^*$, where  $\tau$  is the blow-up time of the system \eqref{Rb1}-\eqref{Rb2}. 
	
	We only state the following theorem and one can use the same ideas as in the proof of Theorem $\ref{t2}$ to obtain the required result. 
	\begin{theorem}\label{Rt2} 
	Assume that the conditions (\ref{Ra4}) hold, and let the initial data be of the form
	\begin{eqnarray} \label{Rq1}
	f_{1}=M_{1} \psi \ \ \mbox{and} \ \ f_{2}=M_{2} \psi ,
	\end{eqnarray} 
	for some positive constants $M_{1}$ and $M_{2}$ with $M_{1} \leq M_{2}$. Then $\sigma_{\ast}\leq\tau$, where $\sigma_{\ast}$ is given by
	
	\begin{align}
	\sigma_{\ast} = \inf \Bigg\{ t\geq 0 : &\int_{0}^{t}\exp\{{\rho_{11}W_{1}(r)+\rho_{12}W_{2}(r)}\}dr
 	\geq\frac{1}{\eta_{1}(\beta_{1}+\gamma_{2}+1)(M^{\beta_{1}}_{1}\left\|\psi \right\|_{\infty}^{\beta_{1}}+M^{\gamma_{2}}_{1}\left\|\psi \right\|_{\infty}^{\gamma_{2}})} \nonumber\\
	\mathrm{ or }\ &\int_{0}^{t} \{\exp\{{\rho_{21}W_{1}(r)+\rho_{22}W_{2}(r)}\}dr\geq  \frac{1}{\eta_{2}(\gamma_{1}+\beta_{2}+1)(M^{\gamma_{1}}_{2}\left\|\psi \right\|_{\infty}^{\gamma_{1}}+M^{\beta_{2}}_{2}\left\|\psi \right\|_{\infty}^{\beta_{2}})}   \Bigg\}. \nonumber
	\end{align}
	\end{theorem}
	The following theorem provides an upper bound for the blow-up in finite time to the system \eqref{Rb1}-\eqref{Rb2} and the proof follows in a similar fashion as in Theorem \ref{thm4.1}.
	\begin{theorem} \label{Rt1} Let $\beta_{1},\beta_{2},\gamma_{1},\gamma_{2}>0$ and for each initial values $f_{i} \geq 0, \ i=1,2,$ we have the following: 
	\begin{itemize}
	\item[1.] Assume that $\beta_{1}=\gamma_{2}=\gamma_{1}=\beta_{2}=m\ (say)>0$. Then $\tau \leq \sigma^{\ast}$, where
	\begin{eqnarray}
	\sigma^{\ast}=\inf \left\lbrace t\geq 0 : \int_{0}^{t} \exp\{{\rho W_{1}(s)+\rho W_{2}(s)}\} ds \geq \frac{2^{(1+m)}}{m(\sigma_{1}+\sigma_{2})}E^{-m}(0) \right\rbrace ,\nonumber
	\end{eqnarray}  
	and $\rho_{11}=\rho_{12}=\rho_{21}=\rho_{22}=\rho$.
	\item[2.] If $\beta_{1}=\gamma_{1}=\beta, \ \beta_{2}=\gamma_{2}=\gamma$ with $\beta>\gamma>0,$ and $D_{1}=\left( \frac{\beta-\gamma}{1+\beta} \right)\left(\frac{1+\beta}{1+\gamma} \right)^{\frac{1+\gamma}{\beta-\gamma}},$
	\begin{align}{\label{nc2}}
	\epsilon_{0}&\leq \min \Bigg\{ 1, \left( h_{1}(0)/D_{1}^{1/1+\gamma} \right)^{\beta-\gamma} \Bigg\}.
	\end{align} 
	Assume that 
	\begin{align}
	2^{-(1+\gamma)}\epsilon_{0}E^{1+\gamma}(0)\geq \epsilon_{0}^{\frac{1+\beta}{\beta-\gamma}}D_{1}.
	\end{align}
	Then $\tau\leq \sigma^{\ast}$, where 
	\begin{align}
	\sigma^{\ast} &= \inf \Bigg\{ t\geq 0 : \int_{0}^{t} \exp\{\rho_{11} W_{1}(s)+\rho_{12} W_{2}(s)\} \wedge \exp\{\rho_{21} W_{1}(s)+\rho_{22} W_{2}(s)\}ds \nonumber\\
	&\hspace{1 in} \left.\geq \left[(\sigma_{1}+\sigma_{2})\gamma E^{\gamma}(0)\left(\frac{\epsilon_0}{2^{1+\gamma}}-\frac{ \epsilon_{0}^{\frac{1+\beta}{\beta-\gamma}}D_{1}}{E^{1+\gamma}(0)}\right)\right]^{-1}\right\}. \nonumber
	\end{align}			
	\item[3.] Let $\beta_{1}>\gamma_{2}>\gamma_{1}>\beta_{2}>0$ and let $D_{2}=\left( \frac{1+\beta_{1}}{1+\gamma_{2}} \right)^{\frac{1+\gamma_{2}}{\beta_{1}-\gamma_{2}}}\frac{\beta_{1}-\gamma_{2}}{1+\beta_{1}}, \ D_{3}=\left( \frac{1+\gamma_{1}}{1+\beta_{2}} \right)^{\frac{1+\beta_{2}}{\gamma_{1}-\beta_{2}}}\frac{\gamma_{1}-\beta_{2}}{1+\gamma_{1}},\\ D_{4}=\frac{\gamma_{2}-\beta_{2}}{1+\gamma_{2}}\left( \frac{1+\gamma_{2}}{1+\beta_{2}} \right)^{\frac{1+\beta_{2}}{\gamma_{2}-\beta_{2}}}$ and \begin{align}\label{vh4} 
	\epsilon_{0}\leq \min \Bigg\{ 1, \left( h_{1}(0)/ D_{2}^{1/1+\gamma_{2}} \right)^{\beta_{1}-\gamma_{2}}, \left( h_{1}(0)/D_{3}^{1/1+\beta_{2}} \right)^{\gamma_{1}-\beta_{2}} \Bigg\}.
	\end{align} 
	Assume that 
	\begin{eqnarray} \label{Rtm2}
	2^{-(1+\gamma_{2})}\epsilon_{0}E^{1+\beta_{2}}(0)(\epsilon_{0}\sigma_{1}+\sigma_{2})\geq  \sigma_{1}\epsilon_{0}^{\frac{1+\beta_{1}}{\beta_{1}-\gamma_{2}}}D_{2} +\sigma_{2} \epsilon_{0}^{\frac{1+\gamma_{1}}{\gamma_{1}-\beta_{2}}}D_{3}+\sigma_{1}\epsilon_{0}^{\frac{ 1+\gamma_{2}}{\gamma_{2}-\beta_{2}}}D_{4}.  
	\end{eqnarray}
	Then $\tau\leq\sigma^{\ast}$, where 
	\begin{align}
	\sigma^{\ast} = \inf \Bigg\{ & t\geq 0 : \int_{0}^{t} \exp\{\rho_{11} W_{1}(s)+\rho_{12} W_{2}(s)\} \wedge \exp\{\rho_{21} W_{1}(s)+\rho_{22} W_{2}(s)\}ds\nonumber\\&\left. \geq \left\{{\beta_{2}E^{\beta_{2}}(0)\left[\frac{(\epsilon_{0}\sigma_{1}+\sigma_{2})\epsilon_{0}}{2^{(1+\gamma_{2})}} -\frac{\left( \sigma_{1}\epsilon_{0}^{\frac{1+\beta_{1}}{\beta_{1}-\gamma_{2}}}D_{2} +\sigma_{2} \epsilon_{0}^{\frac{1+\gamma_{1}}{\gamma_{1}-\beta_{2}}}D_{3}+\sigma_{1}\epsilon_{0}^{\frac{ 1+\gamma_{2}}{\gamma_{2}-\beta_{2}}}D_{4} \right)}{E^{1+\beta_{2}}(0)}  \right]}\right\}^{-1} \right\},  \nonumber
	\end{align}
	where $\rho_{11},\rho_{12},\rho_{21},\rho_{22}$ are defined in \eqref{Ra4} and $E(0)$ is given in \eqref{s1}.
	\end{itemize}
	\end{theorem}
	
	\subsection{Probability of finite time blow-up}
	In this section, we consider the system \eqref{Rb1}-\eqref{Rb2} along with   the set of parameters $\beta_{1},\beta_{2},\gamma_{1},\gamma_{2}>0$ satisying  \eqref{Ra4} and obtain the upper bounds for the probability of non-explosive solution of the system \eqref{Rb1}-\eqref{Rb2}. Further, we choose the other parameters as follows:
	\begin{equation}\label{p1}
	\left\{
	\begin{aligned}
   &V_1=V_2=0,\ C_{11}=C_{12}=C_{21}=C_{22}=1,\\
	&l_{1}:= \frac{k_{11}^{2}+ k_{12}^{2}}{2}, \ l_{2}:= \frac{k_{21}^{2}+ k_{22}^{2}}{2}.   
	\end{aligned}
	\right.
	\end{equation}
    The following result provides a lower bound for the probability of blow-up solution of the system \eqref{Rb1}-\eqref{Rb2}.
	\begin{theorem}\label{thm7.3}
	For each positive initial value $f_{i}\ (i=1,2)$ with $E(0)$ defined in $\eqref{s1},$ and parameters $\epsilon_{0}, D_{1},D_{2},D_{3}$ and $D_{4}$ as defined in Theorem \ref{Rt1} along with \eqref{p1}, we have the following results:
	\begin{itemize}
		\item[1.] If $\beta_{1}=\gamma_{2}=\gamma_{1}=\beta_{2}=m\ (say)$ with $m>0$, then  a lower bound for the probability of blow-up solution of the system \eqref{Rb1}-\eqref{Rb2} is given by 
		\begin{align}
	     \mathbb{P}(\tau<\infty)\geq \mathbb{P}\left\{m E^{m}(0)> 2^{m}\rho^{2} X{(\alpha_{1},1)} \right\},\nonumber 
	     \end{align}
	     whose blow-up time is given by
	     	\begin{eqnarray}
	     	\tau^{\ast}_{1}=\inf \left\lbrace t\geq 0 : \int_{0}^{t} \exp\{\rho W_{1}(t)+\rho W_{2}(t)-\eta ms\} ds \geq 2^{m}{m}^{-1}E^{-m}(0) \right\rbrace, \nonumber
	     	\end{eqnarray} 
	     where $\alpha_{1}=\displaystyle\frac{\eta m }{\rho^{2}},\ \rho_{11}=\rho_{21}=\rho_{12}=\rho_{22}:=\rho.$
		
		\item[2.] If $\beta_{1}=\gamma_{1}=\beta, \ \gamma_{2}=\beta_{2}=\gamma,\ \beta>\gamma>0$ with $\rho_{21}\geq\rho_{11}>0$ and $\rho_{22}\geq \rho_{12}>0$, then a lower bound for the probability of blow-up solution of the system \eqref{Rb1}-\eqref{Rb2} is given by 
	    \begin{align*}
	     \mathbb{P}(\tau<\infty)\geq \frac{8 \eta \gamma}{\rho_{11}^{2}+\rho_{12}^{2}} \sum_{n \geq 1} \frac{\exp \Bigg\{  - \frac{(\rho_{11}^{2}+\rho_{12}^{2})a_{1}}{8}j^{2}_{\frac{2 \eta \gamma}{\rho_{11}^{2}+\rho_{12}^{2}}-1, n} \Bigg\}}{j^{2}_{\frac{2 \eta \gamma}{(\rho_{11}^{2}+\rho_{12}^{2})}-1, n}},\nonumber 
	    \end{align*}
	    	whose blow-up time is given by 
	    	\begin{eqnarray}
	    	\hspace{-0.3 in}\tau^{\ast}_{2} = \inf \left\lbrace  t\geq 0 : \int_{0}^{t}\left(e^{\rho_{11} W_{1}(s)+\rho_{12} W_{2}(s)} \wedge e^{\rho_{21} W_{1}(s)+\rho_{22} W_{2}(s)} \right) e^{-\eta \gamma s}ds \geq a_{1} \right\rbrace,  \nonumber 
	    	\end{eqnarray} 
	    where $a_{1}=\frac{2^{1+\gamma}\gamma^{-1}E(0)}{\epsilon_{0} E^{1+\gamma}(0)-2^{1+\gamma}D_{1}\epsilon_{0}^{\frac{1+\beta}{\beta-\gamma}}}.$

		\item[3.] If $\beta_{1}>\gamma_{2}>\gamma_{1}>\beta_{2}>0,\ \rho_{21}\geq\rho_{11}>0$ and $\rho_{22}\geq \rho_{12}>0$,  
		then a lower bound for the probability of blow-up solution of the system \eqref{Rb1}-\eqref{Rb2} is given by 
	\begin{align}
	 \mathbb{P} (\tau< \infty)&\geq \frac{8 \eta \beta_{2}}{\rho_{11}^{2}+\rho_{12}^{2}} \sum_{n \geq 1} \frac{\exp \Bigg\{  - \frac{(\rho_{11}^{2}+\rho_{12}^{2})a_{1}}{8}j^{2}_{\frac{2 \eta \beta_{2}}{\rho_{11}^{2}+\rho_{12}^{2}}-1, n} \Bigg\}}{j^{2}_{\frac{2 \eta \beta_{2}}{\rho_{11}^{2}+\rho_{12}^{2}}-1, n}},\nonumber 
	\end{align}
    whose blow-up time is given by
	\begin{align}
\tau^{\ast \ast}_{3}&= \inf \Bigg\{ t\geq 0 : \int_{0}^{t} e^{-(\rho_{11} W_{1}(s)+\rho_{12} W_{2}(s)-\eta \beta_{2} s)} \mathbf{1}_{\{\rho_{11}W_{1}(s)+\rho_{12}W_{2}(s)-\eta \gamma s\}} ds \geq a_{2} \Bigg\}, \nonumber 
	\end{align}
	where $\eta=(\lambda+l_{1}+l_{2}),\ a_{2}={{{\beta_{2}}^{-1}E^{-\beta_{2}}(0)\left[ \frac{\epsilon_{0}}{2^{(1+\gamma_{2})}} -\frac{\left(  \epsilon_{0}^{\frac{1+\beta_{1}}{\beta_{1}-\gamma_{2}}}D_{2} + \epsilon_{0}^{\frac{1+\gamma_{1}}{\gamma_{1}-\beta_{2}}}D_{3}+\epsilon_{0}^{\frac{ 1+\gamma_{2}}{\gamma_{2}-\beta_{2}}}D_{4}\right)}{E^{1+\beta_{2}}(0)}\right]^{-1}}},\\
	\mbox{and}\ \ \rho_{11}, \rho_{12},\rho_{21}, \rho_{22}$ are defined in \eqref{Ra4}.
	\end{itemize}
	
	\end{theorem}
	\begin{proof}
	First, we derive the weak formulation of \eqref{Rb1} by using the random transformation $$v_{i}(t,x)=\exp\left\{-k_{i1}W_{1}(t)-k_{i2}W_{2}(t)\right\}u_{i}(t,x), \ \ t\geq 0, \ \ x\in D, \ \ i=1,2.$$ Proceeding in the same way as done in the proof of Theorem $\ref{thm6.1}$ and using \eqref{as6},
	 we get 
	\begin{equation}
	\begin{aligned}
	\frac{\partial v_{i}(t,\psi)}{\partial t} =&-\left(  \lambda+\frac{k_{i1}^{2}+k_{i2}^{2}}{2} \right)v_{i}(t, \psi) \nonumber\\&\quad+\exp\{{-k_{i1}W_{1}(t)-k_{i2}W_{2}(t)}\} \left( e^{k_{i1}W_{1}(t)+k_{i2}W_{2}(t)} v_{i} \right)^{1+\beta_{i}}(t,\psi)\nonumber\\
	& \quad + \exp\{{-k_{i1}W_{1}(t)-k_{i2}W_{2}(t)}\} \left( e^{k_{j1}W_{1}(t)+k_{j2}W_{2}(t)} v_{j} \right)^{1+\gamma_{j}}(t,\psi),\nonumber 
	\end{aligned}
	\end{equation}
	for a.e. $t\geq 0,\ \left\lbrace j \right\rbrace=\left\lbrace1,2\right\rbrace/\{i\}.$	By using Jensen's inequality, we have
	\begin{equation}
	\begin{aligned}
	\frac{\partial v_{i}(t,\psi)}{\partial t} \geq&-\left(  \lambda+\frac{k_{i1}^{2}+k_{i2}^{2}}{2} \right)v_{i}(t, \psi) +\exp\{ \rho_{i1}W_{1}(t)+\rho_{i2}W_{2}(t)\} v_{i}(s,\psi)^{1+\beta_{i}}\\
	& \quad +\exp\{ \rho_{i1}W_{1}(t)+\rho_{i2}W_{2}(t)\} v_{j}(s,\psi)^{1+\gamma_{j}}.
	\end{aligned}
	\end{equation}
Thus, $v_{i}(t,\psi) \geq I_{i}(t), \ t \geq 0, \ i=1,2,$ 	where 
	\begin{equation}
	\left\{
	\begin{aligned}
	\frac{dI_{1}(t)}{dt}&=-\left(\lambda+\frac{k_{11}^{2}+k_{12}^{2}}{2} \right)I_{1}(t)+\exp\{{\rho_{11} W_{1}(t)}+{\rho_{12} W_{2}(t)}\} \left[ I_{1}^{1+\beta_{1}}(t)+I_{2}^{1+\gamma_{2}}(t) \right], \nonumber\\ \frac{dI_{2}(t)}{dt}&=-\left(\lambda+\frac{k_{21}^{2}+k_{22}^{2}}{2} \right)I_{2}(t)+\exp\{{\rho_{21} W_{1}(t)}+{\rho_{22} W_{2}(t)}\} \left[ I_{1}^{1+\beta_{2}}(t)+I_{2}^{1+\gamma_{1}}(t) \right],\\
	I_{i}(0)&=v_{i}(0,\psi), \ i=1,2. 
	\end{aligned}
	\right.
	\end{equation}
	We define $E(t):=I_{1}(t)+I_{2}(t),$ so that $E(\cdot)$ satisfies 
	\begin{align} 
	\frac{dE(t)}{dt}=&-\left(\lambda+l_{1}+l_{2} \right) E(t)+\exp\{{\rho_{11} W_{1}(t)}+{\rho_{12} W_{2}(t)}\} \left[ I_{1}^{1+\beta_{1}}(t)+I_{2}^{1+\gamma_{2}}(t) \right] \nonumber\\
	&\qquad+ \exp\{{\rho_{21} W_{1}(t)}+{\rho_{22} W_{2}(t)}\} \left[ I_{1}^{1+\beta_{2}}(t)+I_{2}^{1+\gamma_{1}}(t) \right]. \nonumber
	\end{align}   
\noindent\textbf{Proof of (1).} If $\beta_{1}=\gamma_{2}=\gamma_{1}=\beta_{2}=m$ (say) and $m>0,$ then $\rho_{11}=\rho_{12}=\rho_{21}=\rho_{22}:=\rho.$ By using the same idea used as in Case 1 of Theorem $\ref{thm6.1},$ we have 
$\tau \leq \tau^{\ast}_{1},$ where $\tau_{1}^{\ast}$ is given by
		\begin{eqnarray}
		\tau^{\ast}_{1}=\inf \left\lbrace t\geq 0 : \int_{0}^{t} \exp\{\rho W_{1}(t)+\rho W_{2}(t)-\eta m s\} ds \geq 2^{m}{m}^{-1}E^{-m}(0) \right\rbrace. \nonumber  
		\end{eqnarray} 
Since $\{(W_1(t),W_{\textcolor{blue}{2}}(t))\}_{t\geq 0}$ is a two-dimensional Brownian motion, we know that the process $\{\rho W_1(t)+\rho W_2(t)\}_{t\geq 0}$ is a Gaussian process with covariance  $2\rho^2\min\{s,t\}$. 		Next, by using the method adopted in $\cite{niu2012}$,  the scaling property of the Brownian motion $\{B(t)\}_{t\geq 0},$   and the proceeding result, we have
	\begin{align}
	\mathbb{P}\{ \tau^{\ast}_{1}=\infty\} &= \mathbb{P}\Bigg( \int_{0}^{\infty} \exp\{\rho W_{1}(t)+\rho W_{2}(t)-\eta m s\}ds\leq 2^{m} {m}^{-1}E^{-m}(0) \Bigg) \nonumber\\
	&=\mathbb{P}\left(\int_{0}^{\infty} \exp\{-\eta m s+\sqrt{2}\rho B(s)\} ds \leq 2^{m} {m}^{-1}E^{-m}(0) \right) \nonumber\\
	&=\mathbb{P}\left(\int_{0}^{\infty} \exp\{-\eta m s +2B(s\rho^{2}/2)\} ds \leq 2^{m} {m}^{-1}E^{-m}(0) \right). \nonumber
	\end{align}  
Let us set $t=\frac{s\rho^{2}}{2}$. The above equality means that there exists a one-dimensional standard Brownian motion $\{B(t)\}_{t\geq 0}$ such that
	\begin{align*}
	\mathbb{P}(\tau<\infty)\geq \mathbb{P}(\tau^{\ast}_{1}< \infty) =1-\mathbb{P}(\tau^{\ast}_{1}=\infty)=1-\mathbb{P}\left(\int_{0}^{\infty} \exp\{2[B(t)-\alpha_{1}t]\} dt\leq \frac{2^{m}\rho^{2}}{2m E^{m}(0)}\right).
	\end{align*}
By using Lemma $\ref{Rll1},$ we have
	\begin{align*}
	\mathbb{P}(\tau<\infty)\geq \mathbb{P}\left\{m E^{m}(0)> 2^{m}\rho^{2} X{(\alpha_{1},1)} \right\},\nonumber 
	\end{align*}
	where $\alpha_{1}=\displaystyle\frac{\eta m }{\rho^{2}}.$
	
	\noindent\textbf{{Proof of (2).}} Suppose $\beta_{1}=\gamma_{1}=\beta, \ \gamma_{2}=\beta_{2}=\gamma$ with $\beta>\gamma>0$. By using the same method as in Case 2 of Theorem $\ref{thm6.1},$ we get $\tau \leq \tau^{\ast}_{2}$ where $\tau_{2}^{\ast}$	
 is given by 
\begin{align}
\hspace{-0.2 in}\tau^{\ast}_{2} = \inf \left\lbrace  t\geq 0 : \int_{0}^{t}\left(e^{\rho_{11} W_{1}(s)+\rho_{12} W_{2}(s)} \wedge e^{\rho_{21} W_{1}(s)+\rho_{22} W_{2}(s)} \right) e^{-\eta \gamma s}ds \geq \frac{2^{1+\gamma}\gamma^{-1}E(0)}{\epsilon_{0} E^{1+\gamma}(0)-2^{1+\gamma}D_{1}\epsilon_{0}^{\frac{1+\beta}{\beta-\gamma}}} \right\rbrace.  \nonumber 
\end{align}
Note that if  $\rho_{21}\geq\rho_{11}>0$ and $\rho_{22}\geq \rho_{12}>0,$ we have
\begin{align}
&\int_{0}^{t} \left( \exp\{{\rho_{11} W_{1}(t)}+{\rho_{12} W_{2}(t)}-\eta \gamma s\} \wedge \exp\{{\rho_{21} W_{1}(t)}+{\rho_{22} W_{2}(t)}-\eta \gamma s\} \right) ds \nonumber\\
&\textcolor{blue}{\geq \int_{0}^{t}\exp\{{\rho_{11} W_{1}(s)+\rho_{12} W_{2}(s)-\eta \gamma s}\} \mathbf{1}_{\{{\rho_{11} W_{1}(t)}+{\rho_{12} W_{2}(t)}-\eta \gamma s \geq 0\}} ds} \nonumber\\
&\quad \textcolor{blue}{+\int_{0}^{t}\exp\{{\rho_{11} W_{1}(s)+\rho_{12} W_{2}(s)-\eta \gamma s}\} \mathbf{1}_{\{{\rho_{11} W_{1}(t)}+{\rho_{12} W_{2}(t)}-\eta \gamma s < 0\}} ds} \nonumber\\ 
&\textcolor{blue}{\geq \int_{0}^{t}\exp\{{\rho_{11} W_{1}(s)+\rho_{12} W_{2}(s)-\eta \gamma s}\} \mathbf{1}_{\{{\rho_{11} W_{1}(t)}+{\rho_{12} W_{2}(t)}-\eta \gamma s \geq 0\}} ds}, \nonumber 
\end{align}
for all $t \geq 0$. We conclude that 
\begin{align}\label{car1}
&\int_{0}^{t} \left( e^{{\rho_{11} W_{1}(t)}+{\rho_{12} W_{2}(t)}-\eta \gamma s} \wedge e^{{\rho_{21} W_{1}(t)}+{\rho_{22} W_{2}(t)}-\eta \gamma s} \right) ds \nonumber\\
&\hspace{0.8 in}\geq \int_{0}^{t}e^{-(\rho_{11} W_{1}(s)+\rho_{12} W_{2}(s)-\eta \gamma s)} \mathbf{1}_{\{\rho_{11}W_{1}(s)+\rho_{12}W_{2}(s)-\eta \gamma s \geq 0\}} ds. 
\end{align}
It follows that 
\begin{align}
\tau^{\ast \ast}_{2}&= \inf \Bigg\{ t\geq 0 : \int_{0}^{t} e^{-(\rho_{11} W_{1}(s)+\rho_{12} W_{2}(s)-\eta \gamma s)} \mathbf{1}_{\{\rho_{11}W_{1}(s)+\rho_{12}W_{2}(s)-\eta \gamma s \geq 0\}} ds \geq a_{1} \Bigg\}, \nonumber 
\end{align}
and $\textcolor{blue}{\tau_{2} \leq} \tau_{2}^{\ast} \leq \tau_{2}^{\ast \ast}.$ Note that $\tau \leq \tau^{\ast \ast}_{2}$, and by definition of $\tau^{\ast\ast}_{2},$ we have
\begin{align}
\mathbb{P} (\tau< \infty) \geq \mathbb{P} (\tau_{2}^{\ast\ast}< \infty) &= \mathbb{P} \Bigg\{ \int_{0}^{\infty} e^{-(\rho_{11} W_{1}(s)+\rho_{12} W_{2}(s)-\eta \gamma s)} \mathbf{1}_{\{\rho_{11}W_{1}(s)+\rho_{12}W_{2}(s)-\eta \gamma s \geq 0\}} ds \geq a_{1} \Bigg\} \nonumber\\
&=\mathbb{P} \Bigg\{ \int_{0}^{\infty} e^{-(-\eta \gamma s+\sqrt{\rho_{11}^{2}+\rho_{12}^{2}}B(s))} \mathbf{1}_{\Big\{\sqrt{\rho_{11}^{2}+\rho_{12}^{2}}B(s)-\eta \gamma s \geq 0\Big\}} ds \geq a_{1} \Bigg\} \nonumber\\
&= \int^{\infty}_{a_{1}} \eta \gamma \sum_{n \geq 1} \exp \Bigg\{ -\left( \frac{\rho_{11}^{2}+\rho_{12}^{2}}{8}j^{2}_{\frac{2 \eta \gamma}{\rho_{11}^{2}+\rho_{12}^{2}}-1, n}\right)y \Bigg\}dy \nonumber\\
&=\frac{8 \eta \gamma}{\rho_{11}^{2}+\rho_{12}^{2}} \sum_{n \geq 1} \frac{\exp \Bigg\{  - \frac{(\rho_{11}^{2}+\rho_{12}^{2})a_{1}}{8}j^{2}_{\frac{2 \eta \gamma}{\rho_{11}^{2}+\rho_{12}^{2}}-1, n} \Bigg\}}{j^{2}_{\frac{2 \eta \gamma}{\rho_{11}^{2}+\rho_{12}^{2}}-1, n}}.\nonumber 
\end{align}


		\noindent \textbf{{Proof of (3).}} Suppose $\beta_{1}>\gamma_{2}>\gamma_{1}>\beta_{2}>0$. By using the same procedure as used in Theorem \ref{thm6.1} (Case 3), we obtain $\tau \leq \tau^{\ast}_{3}$ where $\tau_{3}^{\ast}$
is given by
	\begin{align}
	\tau^{\ast}_{3} = \inf \Bigg\{ t\geq 0 : \int_{0}^{t} &\left( e^{{\rho_{11} W_{1}(t)}+{\rho_{12} W_{2}(t)}} \wedge e^{{\rho_{21} W_{1}(t)}+{\rho_{22} W_{2}(t)}} \right) e^{-\eta \beta_{2}s} ds \geq a_{2}\Bigg\}.  \nonumber
	\end{align}
		Therefore \eqref{car1}, it follows that 
	\begin{align}
	\tau^{\ast \ast}_{3}&= \inf \Bigg\{ t\geq 0 : \int_{0}^{t} e^{-(\rho_{11} W_{1}(s)+\rho_{12} W_{2}(s)-\eta \beta_{2} s)} \mathbf{1}_{\{\rho_{11}W_{1}(s)+\rho_{12}W_{2}(s)-\eta \gamma s \geq 0\}} ds \geq a_{2} \Bigg\}, \nonumber 
	\end{align}
	and $\textcolor{blue}{\tau_{3} \leq} \tau_{3}^{\ast} \leq \tau_{3}^{\ast \ast}.$ Note that $\tau \leq \tau^{\ast \ast}_{3}$, and by definition of $\tau^{\ast\ast}_{3},$ we have 
	\begin{align}
	\mathbb{P} (\tau< \infty) \geq \mathbb{P} (\tau_{3}^{\ast\ast}< \infty) &= \mathbb{P} \Bigg\{ \int_{0}^{\infty} e^{-(\rho_{11} W_{1}(s)+\rho_{12} W_{2}(s)-\eta \beta_{2} s)} \mathbf{1}_{\{\rho_{11}W_{1}(s)+\rho_{12}W_{2}(s)-\eta \gamma s \geq  0\}} ds \geq a_{2} \Bigg\} \nonumber\\
	 &=\mathbb{P} \Bigg\{ \int_{0}^{\infty} e^{-(-\eta \beta_{2} s+\sqrt{\rho_{11}^{2}+\rho_{12}^{2}}B(s))} \mathbf{1}_{\Big\{\sqrt{\rho_{11}^{2}+\rho_{12}^{2}}B(s)-\eta \beta_{2} s \geq 0\Big\}} ds \geq a_{2} \Bigg\} \nonumber\\
	 &= \int^{\infty}_{a_{2}} \eta \beta_{2} \sum_{n \geq 1} \exp \Bigg\{ -\left( \frac{\rho_{11}^{2}+\rho_{12}^{2}}{8}j^{2}_{\frac{2 \eta \beta_{2} }{\rho_{11}^{2}+\rho_{12}^{2}}-1, n}\right)y \Bigg\}dy \nonumber\\
	 &=\frac{8 \eta \beta_{2}}{\rho_{11}^{2}+\rho_{12}^{2}} \sum_{n \geq 1} \frac{\exp \Bigg\{  - \frac{(\rho_{11}^{2}+\rho_{12}^{2})a_{2}}{8}j^{2}_{\frac{2 \eta \beta_{2} }{\rho_{11}^{2}+\rho_{12}^{2}}-1, n} \Bigg\}}{j^{2}_{\frac{2 \eta \beta_{2}}{\rho_{11}^{2}+\rho_{12}^{2}}-1, n}},\nonumber 
	\end{align}
	 which completes the proof. 
	\end{proof}
	
	The following theorem gives an upper bound for the probability of blow-up solution of the system \eqref{b1}-\eqref{b2}.  Let us take \eqref{Ra4} with
	\begin{align}\label{new2}
	\Lambda_{1}= \lambda +\frac{k^{2}_{11} \wedge k^{2}_{12} \wedge k^{2}_{21} \wedge k^{2}_{22}}{2}\ \mbox{and}\ b_{1} =\min \Big\{ \beta_{1}, \beta_{2}, \gamma_{1}, \gamma_{2} \Big\}. 
	\end{align}
	\begin{theorem}
		Assume that $\beta_{1},\beta_{2},\gamma_{1},\gamma_{2}>0$ and $V_{1}=V_{2}=0$. Let $f_{1}=M_{1}\psi$ and $f_{2}=M_{2}\psi,$ for some positive constants $M_{1}$ and $M_{2}$ with $M_{1}\leq M_{2}.$ Then an upper bound for the probability of blow-up solution of the system \eqref{b1}-\eqref{b2} is given by
		\begin{align}
		\mathbb{P}(\tau<\infty) &\leq\mathbb{P} \Big\{ \frac{6}{X(\alpha_{3},1)}\geq \widetilde{N}_4(\rho_{11}+\rho_{21})^{2}\Big\}+\mathbb{P} \Big\{ \frac{6}{X(\alpha_{4},1)}\geq \widetilde{N}_4(\rho_{12}+\rho_{22})^{2}\Big\}\nonumber\\
		&\quad+\mathbb{P} \Big\{ \frac{6}{X(\alpha_{5},1)}\geq \widetilde{N}_4(\rho_{11}+\rho_{12})^{2}\Big\},\nonumber 
		\end{align}
		where
		\begin{align}
		\widetilde{N}_4&= \min\Bigg\{ \frac{1}{\eta_{1}(\beta_{1}+\gamma_{2}+1)(M^{\beta_{1}}_{1}\left\|\psi \right\|_{\infty}^{\beta_{1}}+M^{\gamma_{2}}_{1}\left\|\psi \right\|_{\infty}^{\gamma_{2}})}, \nonumber\\
		&\qquad\qquad\frac{1}{\eta_{2}(\gamma_{1}+\beta_{2}+1)(M^{\gamma_{1}}_{2}\left\|\psi \right\|_{\infty}^{\gamma_{1}}+M^{\beta_{2}}_{2}\left\|\psi \right\|_{\infty}^{\beta_{2}})}\Bigg\}-\frac{1}{\Lambda_{1}b_{1}}, \nonumber \\
		\alpha_{3}&=\displaystyle \frac{4\Lambda_{1}b_{1}}{(\rho_{11}+\rho_{21})^{2}},\ \alpha_{4}=\displaystyle \frac{4\Lambda_{1}b_{1}}{(\rho_{12}+\rho_{22})^{2}},\ \alpha_{3}=\displaystyle \frac{4\Lambda_{1}b_{1}}{(\rho_{11}+\rho_{12})^{2}}\ \mbox{and}\ \Lambda_{1}, b_{1}\ \mbox{are given in}\ \eqref{new2}. \nonumber 
		\end{align}
		\end{theorem}

\begin{proof}
Proceeding similar fashion as in Theorem \ref{thm5.1}, we have
	\begin{align}
	\sigma_{\ast \ast} &= \inf \Bigg\{ t\geq 0 : \int_{0}^{t}\Bigg( \exp\{{\rho_{11}W_{1}(r)+\rho_{12}W_{2}(r)}\} \vee \exp\{{\rho_{21}W_{1}(r)+\rho_{22}W_{2}(r)}\} \Bigg)e^{-\Lambda_{1}b_{1}s}dr\geq
	\nonumber\\&\min\Bigg\{ \frac{1}{\eta_{1}(\beta_{1}+\gamma_{2}+1)(M^{\beta_{1}}_{1}\left\|\psi \right\|_{\infty}^{\beta_{1}}+M^{\gamma_{2}}_{1}\left\|\psi \right\|_{\infty}^{\gamma_{2}})}, \frac{1}{\eta_{2}(\gamma_{1}+\beta_{2}+1)(M^{\gamma_{1}}_{2}\left\|\psi \right\|_{\infty}^{\gamma_{1}}+M^{\beta_{2}}_{2}\left\|\psi \right\|_{\infty}^{\beta_{2}})}\Bigg\} \Bigg\}. \nonumber
	\end{align}
It is clear that $\sigma_{\ast \ast} \leq \sigma_{\ast}\leq \tau.$ Note that if  $\rho_{11}\geq\rho_{21}>0$ and $\rho_{12}\geq \rho_{22}>0,$ we have
\begin{align}
&\int_{0}^{t} \left( \exp\{{\rho_{11} W_{1}(t)}+{\rho_{12} W_{2}(t)}-\Lambda_{1}b_{1} s\} \vee \exp\{{\rho_{21} W_{1}(t)}+{\rho_{22} W_{2}(t)}-\Lambda_{1}b_{1} s\} \right) ds \nonumber\\
&= \int_{0}^{t}\exp\{\rho_{21} W_{1}(s)+\rho_{22} W_{2}(s)-\Lambda_{1}b_{1} s\} \mathbf{1}_{\{W_{1}<0,W_{2}<0\}} ds \nonumber\\
&\quad+\int_{0}^{t}\exp\{{\rho_{11} W_{1}(s)+\rho_{12} W_{2}(s)-\Lambda_{1}b_{1} s}\}\vee \exp\{{\rho_{21} W_{1}(s)+\rho_{22} W_{2}(s)-\Lambda_{1}b_{1} s}\}\mathbf{1}_{\{W_{1}\geq 0,W_{2} < 0\}} ds \nonumber\\
&\quad+\int_{0}^{t}\exp\{{\rho_{11} W_{1}(s)+\rho_{12} W_{2}(s)-\Lambda_{1}b_{1} s}\}\vee \exp\{{\rho_{21} W_{1}(s)+\rho_{22} W_{2}(s)-\Lambda_{1}b_{1} s}\}\mathbf{1}_{\{W_{1}< 0,W_{2} \geq 0\}} ds \nonumber\\
&\quad+\int_{0}^{t}\exp\{{\rho_{11} W_{1}(s)+\rho_{12} W_{2}(s)-\Lambda_{1}b_{1} s}\} \mathbf{1}_{\{W_{1} \geq 0,W_{2} \geq 0\}} ds \nonumber\\
&\leq \int_{0}^{\infty}\exp\{{-\Lambda_{1}b_{1} s}\} ds +\int_{0}^{t}\exp\{{(\rho_{11}+\rho_{21})W_{1}(s)-\Lambda_{1}b_{1} s}\} ds \nonumber\\
&\quad+\int_{0}^{t}\exp\{{(\rho_{12}+\rho_{22}) W_{2}(s)-\Lambda_{1}b_{1} s}\} ds+\int_{0}^{t}\exp\{{\rho_{11} W_{1}(s)+\rho_{12} W_{2}(s)-\Lambda_{1}b_{1} s}\} ds \nonumber\\
&=\frac{1}{\Lambda_{1}b_{1}}+\int_{0}^{t}\exp\{{(\rho_{11}+\rho_{21})W_{1}(s)-\Lambda_{1}b_{1} s}\} ds+\int_{0}^{t}\exp\{{(\rho_{12}+\rho_{22}) W_{2}(s)-\Lambda_{1}b_{1} s}\} ds\nonumber\\&\quad+\int_{0}^{t}\exp\{{\rho_{11} W_{1}(s)+\rho_{12} W_{2}(s)-\Lambda_{1}b_{1} s}\} ds. \nonumber
\end{align}
It follows that 
\begin{align}
\sigma_{\ast \ast \ast}&= \inf \Bigg\{ t\geq 0 : \int_{0}^{t} \Big\{ \exp\{{(\rho_{11}+\rho_{21})W_{1}(s)-\Lambda_{1}b_{1}s}\}+\exp\{{(\rho_{12}+\rho_{22}) W_{2}(s)-\Lambda_{1}b_{1}s}\} \nonumber\\ &\hspace{1.0 in}+\exp\{{\rho_{11} W_{1}(s)+\rho_{12} W_{2}(s)-\Lambda_{1}b_{1}s}\} \Big\} ds \geq \widetilde{N}_4 \Bigg\}, \nonumber 
\end{align}
and $\sigma_{\ast \ast \ast} \leq\sigma_{\ast \ast} \leq \tau.$ Then by definition of $\sigma_{\ast \ast \ast},$ we have
\begin{align}
\mathbb{P} (\sigma_{\ast \ast \ast}< \infty) &\leq \mathbb{P} \Bigg\{ \int_{0}^{\infty} \Big\{ \exp\{{(\rho_{11}+\rho_{21})W_{1}(s)-\Lambda_{1}b_{1} s}\} +\exp\{{(\rho_{12}+\rho_{22}) W_{2}(s)-\Lambda_{1}b_{1} s}\} \nonumber\\ &\hspace{1.0 in}+\exp\{{\rho_{11} W_{1}(s)+\rho_{12} W_{2}(s)-\Lambda_{1}b_{1} s}\} \geq \widetilde{N}_4 \Big\} ds \Bigg\} \nonumber\\
& \leq \mathbb{P} \Bigg\{ \int_{0}^{\infty} \exp\{{(\rho_{11}+\rho_{21})W_{1}(s)-\Lambda_{1}b_{1} s}\}ds \geq \frac{\widetilde{N}_4}{3}, \nonumber\\&\qquad \int_{0}^{\infty} \exp\{{(\rho_{12}+\rho_{22}) W_{2}(s)-\Lambda_{1}b_{1} s}\} ds \geq \frac{\widetilde{N}_4}{3}, \nonumber\\
&\qquad \int_{0}^{\infty} \exp\{{\rho_{11} W_{1}(s)+\rho_{12} W_{2}(s)-\Lambda_{1}b_{1} s}\} ds \geq \frac{\widetilde{N}_4}{3} \Bigg\} \nonumber \\
& \leq \mathbb{P} \Bigg\{ \int_{0}^{\infty} \exp\{{(\rho_{11}+\rho_{21})W_{1}(s)-\Lambda_{1}b_{1} s}\}ds \geq \frac{\widetilde{N}_4}{3}\Bigg\}\nonumber\\&\qquad+\mathbb{P} \Bigg\{\int_{0}^{\infty} \exp\{{(\rho_{12}+\rho_{22}) W_{2}(s)-\Lambda_{1}b_{1} s}\} ds\geq \frac{\widetilde{N}_4}{3}\Bigg\} \nonumber\\
&\qquad+\mathbb{P} \Bigg\{\int_{0}^{\infty} \exp\{{\rho_{11} W_{1}(s)+\rho_{12} W_{2}(s)-\Lambda_{1}b_{1} s}\}ds \geq \frac{\widetilde{N}_4}{3} \Bigg\}. \nonumber  	
\end{align}
Therefore, we have 
\begin{align}
&\mathbb{P} \Bigg\{ \int_{0}^{\infty} \exp\{{(\rho_{11}+\rho_{21})W_{1}(s)-\Lambda_{1}b_{1} s}\}ds\geq \frac{\widetilde{N}_4}{3}\Bigg\}\nonumber\\&=	\mathbb{P} \Bigg\{ \int_{0}^{\infty} \exp\{{2W_{1}\left( \frac{(\rho_{11}+\rho_{21})^{2}s}{4}\right) -\Lambda_{1}b_{1} s}\}ds\geq \frac{\widetilde{N}_4}{3}\Bigg\},\nonumber
\end{align}
by using the scaling property of $\{W_1(t)\}_{t\geq 0}.$ A transformation $t=\frac{(\rho_{11}+\rho_{21})^{2}s}{4}$ yields that 
\begin{align}
&\mathbb{P} \Bigg\{ \int_{0}^{\infty} \exp\{{(\rho_{11}+\rho_{21})W_{1}(s)-\Lambda_{1}b_{1} s}\}ds\geq \frac{\widetilde{N}_4}{3}\Bigg\}\nonumber\\&=	\mathbb{P} \Bigg\{ \frac{4}{(\rho_{11}+\rho_{21})^{2}}\int_{0}^{\infty} e^{2W_{1}(t)-\frac{4\Lambda_{1}b_{1} t}{(\rho_{11}+\rho_{21})^{2}} }dt\geq \frac{\widetilde{N}_4}{3}\Bigg\} \nonumber\\
&=	\mathbb{P} \Bigg\{\frac{6}{X(\alpha_{3},1)}\geq \widetilde{N}_4(\rho_{11}+\rho_{21})^{2}\Bigg\}. \nonumber
\end{align}
Similarly, we have 
	\begin{align*}
	\mathbb{P} \Bigg\{ \int_{0}^{\infty} \exp\{{(\rho_{12}+\rho_{22})W_{2}(s)-\Lambda_{1}b_{1} s}\}ds<\frac{x}{3}\Bigg\}&=\mathbb{P} \Bigg\{\frac{6}{X(\alpha_{4},1)}\geq\widetilde{N}_4(\rho_{12}+\rho_{22})^{2}\Bigg\}
\end{align*}and 
\begin{align*}
	\mathbb{P} \Bigg\{\int_{0}^{\infty} \exp\{{\rho_{11} W_{1}(s)+\rho_{12} W_{2}(s)-\Lambda_{1}b_{1} s}\}ds <\frac{x}{3} \Bigg\}&= \mathbb{P} \Bigg\{\frac{6}{X(\alpha_{5},1)}\geq\widetilde{N}_4(\rho_{11}+\rho_{12})^{2}\Bigg\}. \nonumber 
	\end{align*}
	Therefore, we obtain
	\begin{align}
	\mathbb{P}(\tau<\infty) &\leq \mathbb{P} (\sigma_{\ast\ast \ast}< \infty)\nonumber\\&\leq\mathbb{P} \Big\{ \frac{6}{X(\alpha_{3},1)}\geq \widetilde{N}_4(\rho_{11}+\rho_{21})^{2}\Big\}+\mathbb{P} \Big\{ \frac{6}{X(\alpha_{4},1)}\geq \widetilde{N}_4(\rho_{12}+\rho_{22})^{2}\Big\}\nonumber\\
	&\qquad+\mathbb{P} \Big\{ \frac{6}{X(\alpha_{5},1)}\geq \widetilde{N}_4(\rho_{11}+\rho_{12})^{2}\Big\},\nonumber 
	\end{align}
 which completes the proof.
\end{proof}

	\begin{Rem}
		\begin{enumerate}
			\item  Taking $k_1=k_2=0$ in \eqref{b1}-\eqref{b2}, we have $v_i=u_i,\ i=1,2$ (see \eqref{1}) and, moreover, if $V_1=V_{2}=0$ and $\beta_1=\beta_2=\gamma_1=\gamma_2=\beta$, we obtain from \eqref{tau1} that $$\tau_1^{\ast}=\inf\Bigg\{t\geq 0:\int_{0}^{t}e^{-\mu \beta s} ds \geq 2^{\beta}{\beta}^{-1}E^{-\beta}(0)\Bigg\}.$$  Therefore $\mathbb{P}(\tau_1^{*}=+\infty)=0$  provided 
			\begin{align}
				\int_{0}^{\infty} \exp\{ -\mu \beta s \} ds &\geq 2^{\beta} \beta^{-1}E^{-\beta}(0),\ \text{ that is, }\ 
				\frac{1}{\mu \beta} \geq 2^{\beta} \beta^{-1}E^{-\beta}(0),\nonumber
			\end{align}
			where $\mu=\lambda$ by \eqref{s1} and we have 
			\begin{align}
				\int_{D} [f_{1}(x)+f_{2}(x)] \psi(x)dx &\geq 2 \lambda^{\frac{1}{\beta}}. \nonumber
			\end{align}
			Thus, the problem \eqref{b1}-\eqref{b2} reduces to the Fujita problem.
			\item Taking $C_{11}=C_{22}=0, C_{12}=C_{21}=1$ in \eqref{Rb1}-\eqref{Rb2}, the results of Theorem 7.1 and case 1, case 2 of Theorem 7.2 reduce to the results of Theorem 3.1 and Theorem 3.4 respectively of \cite{li}. Moreover, case 2 of Theorem 7.3 reduce to Theorem 4.2 of \cite{li}.
			\item Due to the relation \eqref{a4},  the blow-up time $\tau_{\ast}$ of the system \eqref{b1}-\eqref{b2} obtained in Section \ref{sec3}  (cf. \eqref{e1}) is a particular case of the blow-up time $\tau_{\ast \ast}$ established in Section \ref{sec5} (cf. \eqref{e2}). Also, in Section \ref{sec4}, under the assumption \eqref{a4},  the blow-up time $\tau^{\ast}$ of the system \eqref{b1}-\eqref{b2} obtained in Theorem \ref{thm4.1} is a special case of the blow-up time $\tau^{\ast \ast}$ established  in Theorem \ref{thm5.2} of Section \ref{sec5} whereas in Section \ref{sec7}, we have obtained the lower and upper bounds for the finite time blow-up of solutions for the semilinear system \eqref{b1}-\eqref{b2} perturbed by two-dimensional Brownian motion. 
			
		\end{enumerate}

	\end{Rem}
	
	\section{Conclusion}
	In this article, we estimated lower and upper bounds for the blow-up times and also obtained the bounds for the probability of blow-up solutions of a system of semilinear SPDEs with a standard one-dimensional Brownian motion. Then, we extended the above results to the same system perturbed by a two-dimensional Brownian motion. An extension of  the above results to semilinear SPDEs  subjected to  multiplicative noise of power type (cf. \cite{lv} for similar problems), will be a hard and challenging problem. We are transforming a stochastic PDE into a random PDE and such a transformation is known only for additive or linear multiplicative noises. Therefore, the blow-up problem for semilinear SPDEs with  multiplicative noise of power type  will be addressed in a future work.
	\medskip\noindent
	
\noindent{\bf Acknowledgements:} The First author is supported by the University Research Fellowship of Periyar University, India. M. T. Mohan would  like to thank the Department of Science and Technology (DST), India for Innovation in Science Pursuit for Inspired Research (INSPIRE) Faculty Award (IFA17-MA110). The Third author is supported by the Fund for Improvement of Science and Technology Infrastructure (FIST) of DST (SR/FST/MSI-115/2016). The authors sincerely would like to thank the reviewers for their valuable comments and suggestions, which helped us to improve the manuscript significantly. \\
	
	\noindent {\bf Data availability:} Data sharing not applicable to this article as no datasets were generated or analysed during the current study.\\

	\noindent {\bf Disclosure statement:} No potential competing interest was reported by the authors.

\end{document}